\documentclass[10pt]{article}

\usepackage{amsfonts}
\usepackage{amssymb}
\usepackage{amsthm}
\usepackage{amsmath}
\usepackage{bm}
\usepackage{amscd} 
\usepackage{graphicx}
\usepackage{mathrsfs}
\usepackage{amsrefs}
\usepackage{comment}
\usepackage[all]{xy}
\usepackage{MnSymbol}

\usepackage{comment}

\usepackage{fullpage}
\usepackage{hyperref}

\numberwithin{equation}{section} \DeclareMathSizes{2}{10}{12}{13}
\newtheorem{theorem}{Theorem}[section]
\newtheorem{lemma}[theorem]{Lemma}

\newtheorem{prop}[theorem]{Proposition}

\newtheorem{remark}[theorem]{Remark}
\newtheorem{definition}[theorem]{Definition}

\numberwithin{equation}{section}

\begin{document}
\title{Categorified Fredholm Modules and Chern Characters}
\author{Mamta Balodi\footnote{Email: mamta.balodi@gmail.com}~ and Abhishek Banerjee\footnote{Email: abhishekbanerjee1313@gmail.com} \footnote{§AB was partially supported by SERB Matrics fellowship MTR/2017/000112
}}
\date{}
\maketitle
\centerline{\emph{Department of Mathematics, Indian Institute of Science, Bangalore - 560012, India.}}

\begin{abstract} In this paper, we continue  our program of systematic categorification of the Noncommutative Differential Geometry of Connes. We replace a ring with a small $\mathbb C$-linear category, seen as a ring with several objects in the sense of Mitchell. We introduced Fredholm modules over this category and construct a Chern character taking values in the cyclic cohomology of $\mathcal C$. We show that this categorified Chern character depends only on the homotopy
class of the Fredholm module and is well-behaved
with respect to the periodicity operator in cyclic cohomology.
\end{abstract}

\emph{{\bf \emph{MSC(2010) Subject Classification:}} 18F25, 19D55, 53C99}

\medskip

\emph{{\bf \emph{Keywords:} } Categorified Fredholm modules, Categorified Chern characters. }

\section{Introduction}

In this paper, we continue  our program of systematic categorification of the Noncommutative Differential Geometry of Connes \cite{C2}. We replace a ring with a small $\mathbb C$-linear category, seen as a ring with several objects in the sense of Mitchell \cite{Mit0}. In \cite{BB1}, we have described the 
cocycles and coboundaries
that determine the cyclic cohomology $H^\bullet_\lambda(\mathcal C)$  by extending Connes' original construction of cyclic cohomology in terms of
cycles and closed graded traces on differential graded algebras. The key objective of \cite{BB1} is
to describe cyclic cocycles and coboundaries in terms of closed graded traces over a differential graded  semicategory. In fact, we showed
in \cite{BB1} that each  cyclic cocycle in $Z^n_\lambda(\mathcal C)$ is obtained as the character of a $n$-dimensional cycle
over $\mathcal C$, while every cyclic coboundary in $B^n_\lambda(\mathcal C)$
 is given by the character of a ``vanishing
cycle.'' This description of cyclic cohomology is best suited for the study of `categorified Fredholm modules'
that we now develop in this paper.

\smallskip
In the Noncommutative Differential Geometry of Connes \cite{C2}, an $n$-summable Fredholm module over an algebra
$A$ leads to a class in the $n$-th cyclic cohomology of $A$. In other words, we have a Chern character
\begin{equation*}
ch^* :\{\mbox{$n$-summable Fredholm modules over $A$}\} \longrightarrow  H_\lambda^n
(A)
\end{equation*}
This suggests that we should consider Fredholm modules over a small $\mathbb C$-linear category $\mathcal C$. This would consist
of a functor on $\mathcal C$ taking values in the category $SHilb$ of separable Hilbert spaces and bounded linear maps (see Definition \ref{dD3.1}). In this paper, we proceed in a manner analogous to Connes \cite{C2} to show that an $n$-summable Fredholm module over $\mathcal C$ leads to a cycle
over $\mathcal C$ and hence to a class in the cyclic cohomology of $\mathcal C$. We show that this categorified Chern character is well behaved with respect to the periodicity operator on the cyclic cohomology of $\mathcal C$. We also show that the categorified Chern character depends only on the homotopy
class of the Fredholm module. 

\smallskip
The literature on cyclic homology and Chern characters is vast and we refer the reader, for instance, to \cite{C1}, \cite{C3}, \cite{GSz}, \cite{GS}, \cite{HSS}, \cite{Kar1}, \cite{Kar2},  \cite{LQ}, \cite{Loday}, \cite{carthy} for more on this subject.

\smallskip
We fix here some notation that we will use throughout this paper. We will denote by $(C^\bullet(\mathcal C),b)$ the standard Hochschild complex
of the dual of the cyclic nerve of $\mathcal C$. The modified Hochschild differential (with the last face operator missing) will be denoted by $b'$. 
Accordingly, the Hochschild cocycles and Hochschild coboundaries will be denoted by $Z^\bullet(\mathcal C)$ and $B^\bullet(\mathcal C)$ respectively.

\smallskip
We will denote by $\tau_n$ the unsigned cyclic operator on $C^n(\mathcal C)$ and by $\lambda_n$ the signed cyclic operator $(-1)^n\tau_n$ on $C^n(\mathcal C)$. 
The complex computing cyclic cohomology will be denoted by $C^\bullet_\lambda := Ker(1-\lambda)$. Accordingly, the cyclic cocycles and cyclic coboundaries will be denoted by $Z^\bullet_\lambda(\mathcal C)$ and $B^\bullet_\lambda(\mathcal C)$ respectively.

\section{Even Fredholm modules over categories}\label{evencycl}

Throughout, we let $\mathcal C$ be a small $\mathbb C$-linear category. Our categorified Fredholm modules will consist of functors from $\mathcal C$  taking values in 
separable Hilbert spaces.  Let $SHilb$ be the category whose objects are separable Hilbert spaces and  whose morphisms are bounded linear maps.  

\smallskip

 Given separable Hilbert spaces $\mathcal H_1$ and $\mathcal H_2$, let $\mathcal{B}(\mathcal H_1,\mathcal H_2)$ denote the space of all bounded linear operators from $\mathcal H_1$ to $\mathcal H_2$  and $\mathcal{B}^\infty(\mathcal H_1,\mathcal H_2) \subseteq \mathcal{B}(\mathcal H_1,\mathcal H_2)$ be the space of all  compact operators. For any bounded operator  $T \in \mathcal{B}(\mathcal H_1,\mathcal H_2)$, let $\mu_n(T)$ denote the $n$-th singular value of $T$. In other words, $\mu_n(T)$ is the $n$-th (arranged
 in decreasing order) eigenvalue of the positive operator  $|T|:=(T^*T)^{\frac{1}{2}}$. 
For $1 \leq p < \infty$, the $p$-th Schatten class is defined to be the space
\begin{equation*}
\mathcal{B}^p(\mathcal H_1,\mathcal H_2):=\{T \in \mathcal{B}(\mathcal H_1,\mathcal H_2)~|~ \sum \mu_n(T)^p < \infty\}
\end{equation*}
Clearly, $\mathcal{B}^p(\mathcal H_1,\mathcal H_2) \subseteq  \mathcal{B}^q(\mathcal H_1,\mathcal H_2)$  for $p \leq q$. For $p=1$, the space $\mathcal{B}^1(\mathcal H_1,\mathcal H_2)$ is the collection of all trace class operators from $\mathcal H_1$ to $\mathcal H_2$. For  $T \in \mathcal{B}^1(\mathcal H_1,\mathcal H_2)$, we write $Tr(T):=\sum \mu_n(T)$. It is well known that
\begin{equation}\label{commutr}
Tr(T_1T_2)=Tr(T_2T_1) \quad \forall T_1 \in \mathcal{B}^{n_1}(\mathcal{H}, \mathcal{H}'),~ T_2 \in \mathcal{B}^{n_2}(\mathcal{H}', \mathcal{H}), \frac{1}{n_1}+\frac{1}{n_2}=1
\end{equation}

We note that $\mathcal{B}^p(\mathcal H_1,\mathcal H_2)$ is an ``ideal" in $SHilb$ in the following sense: consider the functor
\begin{equation*}
\begin{array}{c} \mathcal{B}(-,-):SHilb^{op} \otimes SHilb \longrightarrow Vect_\mathbb{C} \qquad \mathcal B(-,-)(\mathcal H_1,\mathcal H_2):=\mathcal{B}(\mathcal H_1,\mathcal H_2)
\\ \mathcal{B}(-,-)(\phi_1,\phi_2):\mathcal{B}(\mathcal H_1, \mathcal H_2) \longrightarrow \mathcal{B}(\mathcal H_1',\mathcal H_2')\qquad T \mapsto \phi_2T\phi_1\\
\end{array}
\end{equation*}  taking values in the category $Vect_{\mathbb C}$ of $\mathbb C$-vector spaces. Then, $\mathcal{B}^p(-,-)$ is a subfunctor of $\mathcal{B}(-,-)$. In other words, for  morphisms $\phi_1:\mathcal H_1'\longrightarrow \mathcal H_1$, $\phi_2:\mathcal H_2\longrightarrow \mathcal H_2'$  and any $T \in \mathcal{B}^p(\mathcal H_1,\mathcal H_2)$, we have $\phi_2T\phi_1 \in \mathcal{B}^p(\mathcal H_1',\mathcal H_2')$.

\smallskip
We fix the following convention for the commutator notation:  Let $\mathscr{H}:\mathcal{C}  \longrightarrow SHilb$ be a functor and $\mathcal{G}:=\{\mathcal{G}_X:\mathscr{H}(X) \longrightarrow \mathscr{H}(X)\}_{X \in Ob(\mathcal{C})}$ be a collection of bounded linear operators. Then, we set 
\begin{equation*}
[\mathcal G,-]:\mathcal B(\mathscr H(X),\mathscr H(Y))\longrightarrow \mathcal B(\mathscr H(X),\mathscr H(Y)) \qquad [\mathcal G,T]:= \mathcal{G}_Y \circ T - T\circ \mathcal{G}_X\in \mathcal B(\mathscr H(X),\mathscr H(Y))
\end{equation*} 
We now let $SHilb_{\mathbb{Z}_2}$ be the category whose objects are $\mathbb{Z}_2$-graded separable Hilbert spaces and whose morphims are bounded linear maps.   Let $\mathscr{H}:\mathcal{C}  \longrightarrow SHilb_{\mathbb{Z}_2}$ be a functor and $\mathcal{G}:=\{\mathcal{G}_X:\mathscr{H}(X) \longrightarrow \mathscr{H}(X)\}_{X \in Ob(\mathcal{C})}$ be a collection of bounded linear operators of the same degree $|\mathcal G|$. Then, we set 
\begin{equation}\label{gradcomm}
\begin{array}{c}
[\mathcal G,-]:\mathcal B(\mathscr H(X),\mathscr H(Y))\longrightarrow \mathcal B(\mathscr H(X),\mathscr H(Y)) \\  
\mbox{$[\mathcal G,T]:= \mathcal{G}_Y \circ T -(-1)^{|\mathcal G||T|} T\circ \mathcal{G}_X\in \mathcal B(\mathscr H(X),\mathscr H(Y))$}\\
\end{array}
\end{equation}

\begin{definition}\label{dD3.1}
Let $\mathcal{C}$ be a small $\mathbb{C}$-category and let $p \in [1,\infty)$. We consider a pair $(\mathscr H,\mathcal F)$ as follows.

\smallskip
\begin{itemize}
\item[(1)] A functor  $\mathscr{H}:\mathcal{C}  \longrightarrow SHilb_{\mathbb{Z}_2}$ such that $\mathscr H(f):\mathscr H(X)\longrightarrow \mathscr H(Y)$ is a linear operator
of degree $0$ for each  $f \in Hom_{\mathcal{C}}(X,Y)$.

\item[(2)] A collection $\mathcal{F}:=\{\mathcal{F}_X:\mathscr{H}(X) \longrightarrow \mathscr{H}(X)\}_{X \in Ob(\mathcal{C})}$  of bounded linear operators of degree $1$ such that $\mathcal{F}_X^2=id_{\mathscr{H}(X)}$ for each $X \in Ob(\mathcal{C})$.
\end{itemize}

The pair $(\mathscr{H},\mathcal F)$ is said to be a $p$-summable even Fredholm module over the category $\mathcal{C}$ if
every $f \in Hom_{\mathcal{C}}(X,Y)$ satisfies
\begin{equation}\label{stneve}
[\mathcal{F},f]:= \left(\mathcal{F}_Y \circ \mathscr{H}(f) - \mathscr{H}(f) \circ \mathcal{F}_X \right)     \in \mathcal{B}^p\left(\mathscr{H}(X),\mathscr{H}(Y)\right)
\end{equation}
\end{definition} 

 We first recall from \cite{BB1}  the construction of universal DG-semicategories and the notion of closed graded traces on DG-semicategories. In essence, a semicategory is a ``category without units.'' 

\begin{definition}(see Mitchell \cite[$\S$ 4]{Mit})
A semicategory $\mathcal D$ consists of a collection $Ob(\mathcal{D})$ of objects together with a set of morphisms $Hom_{\mathcal{D}}(X,Y)$ for each $X,Y \in Ob(\mathcal{D})$ and an associative composition. A semifunctor $F:\mathcal{D} \longrightarrow \mathcal{D}'$ between semicategories assigns to each $X \in Ob(\mathcal{D})$ an object $F(X) \in Ob(\mathcal{D}')$, 
a morphism $F(f) \in Hom_{\mathcal D'}(F(X),F(Y))$ to each $f \in Hom_\mathcal{D}(X,Y)$ and preserves composition.
\end{definition}

It is clear that any ordinary category may be treated as a semicategory. On the other hand, to any $\mathbb{C}$-semicategory $\mathcal{D}$, we can associate an ordinary $\mathbb{C}$-category
$\tilde{\mathcal{D}}$ by adjoining  unit morphisms as follows:
\begin{equation}\label{addnit}
\begin{array}{ll}
Ob(\tilde{\mathcal{D}}):&=Ob(\mathcal{D})\\
Hom_{\tilde{\mathcal{D}}}(X,Y):&=\left\{\begin{array}{ll} Hom_\mathcal{D}(X,X) \bigoplus \mathbb C &  \mbox{if $X=Y$}\\
Hom_\mathcal{D}(X,Y) & \mbox{if $X \neq Y$} \\ \end{array}\right.
\end{array}
\end{equation}

\begin{definition} (see \cite[Definition 5.2]{BB1})
A differential graded semicategory (DG-semicategory) $(\mathcal{S},\hat\partial)$ is a $\mathbb{C}$-semicategory $\mathcal{S}$ such that
\begin{itemize}
\item[(i)] $Hom^\bullet_\mathcal{S}(X,Y)=\big(Hom^n_\mathcal{S}(X,Y),\hat\partial^n_{XY}\big)_{n \geq 0}$ is a cochain complex of $\mathbb C$-vector spaces for each $X,Y \in Ob(\mathcal{S})$.

\item[(ii)] the composition map
$$Hom^\bullet_\mathcal{S}(Y,Z) \otimes Hom^\bullet_\mathcal{S}(X,Y) \longrightarrow Hom^\bullet_\mathcal{S}(X,Z)$$
is a morphism of complexes. Equivalently,
\begin{equation}
\hat\partial^n_{XZ}(g\circ f)=\hat\partial^{n-r}_{YZ}(g)\circ f+(-1)^{n-r}g\circ \hat\partial^{r}_{XY}(f) \label{comp}
\end{equation}
for any $f \in Hom_\mathcal{S}^r(X,Y)$ and $g \in Hom_\mathcal{S}^{n-r}(Y,Z)$.
\end{itemize} We will often drop the subscript and simply write $\hat\partial^\bullet$ for the differential
on any $Hom^\bullet_{\mathcal S}(X,Y)$. 

\smallskip
A DG-semifunctor $\alpha:(\mathcal{S},\hat\partial)\longrightarrow (\mathcal{S}',\hat\partial')$    is a $\mathbb C$-linear semifunctor $\alpha:\mathcal{S} \longrightarrow \mathcal{S}'$ such that the induced map $Hom^\bullet_\mathcal{S}(X,Y) \longrightarrow Hom^\bullet_{\mathcal{S}'}(\alpha X,\alpha Y)$, $f \mapsto \alpha (f)$, is a morphism of complexes for each $X,Y \in Ob(\mathcal{S})$. 
\end{definition}

If $(\mathcal S,\hat\partial)$ is a DG-semicategory, it is clear that $\mathcal S^0$ is an ordinary semicategory. 
We now recall the construction of the universal DG-semicategory $\Omega \mathcal C$ associated to a small $\mathbb{C}$-category $\mathcal{C}$ (see, \cite[Proposition 5.5]{BB1}):
\begin{equation}\label{xudga} 
\begin{array}{c}
Ob(\Omega\mathcal{C}):=Ob(\mathcal{C})\qquad Hom^\bullet_{\Omega \mathcal C}(X,Y)=\left(Hom^n_{\Omega \mathcal C}(X,Y),\partial^n_{XY}\right)_{n\geq 0}\\
Hom^n_{\Omega\mathcal{C}}(X,Y):=\underset{(X_1,...,X_n)\in Ob(\mathcal C)^n}{\bigoplus}Hom_{\tilde{\mathcal C}}(X_1,Y)\otimes Hom_{\mathcal C}
(X_2,X_1)\otimes \dots \otimes Hom_{\mathcal C}(X,X_n)
\end{array}
\end{equation}  An element of the form $\tilde{f}^0\otimes f^1\otimes ...\otimes f^n$ in $Hom^n_{\Omega\mathcal{C}}(X,Y)$ will be denoted by $\tilde{f}^0df^1 \ldots df^n$. Using 
\eqref{addnit}, a morphism 
$\tilde{f}^0$ in $Hom_{\tilde{\mathcal C}}(X_1,Y)$ will frequently be written as $f^0+\mu$, where $\mu\in \mathbb C$. 
The composition in $\Omega\mathcal{C}$ is determined by
\begin{equation}\label{compoDG}
 \begin{array}{l}
((f^0+\mu) df^1\dots df^n)\circ ((f^{n+1}+\mu')df^{n+2}\dots df^m)\\=\sum\limits_{j=1}^n (-1)^{n-j} (f^0+\mu)df^1\dots d(f^jf^{j+1}) \dots df^m + (-1)^n (f^0+\mu)f^1df^2 \dots df^m\hspace{1.4in}\\
\textrm{ }+\mu' (f^0+\mu)df^1\dots df^ndf^{n+2}\dots df^m\\
\end{array}
\end{equation}
 In particular, we obtain
\begin{equation}\label{symb}
 f^0df^1\dots df^n=f^0\circ df^1\circ \dots \circ df^n \qquad d(f^0f^1)=(df^0)f^1+f^0(df^1) 
\end{equation}
For each $X,Y \in Ob(\Omega\mathcal{C})$,  the differential $\partial^n_{XY}:Hom^n_{\Omega\mathcal{C}}(X,Y) \longrightarrow Hom^{n+1}_{\Omega\mathcal{C}}(X,Y)$ is given by 
\begin{equation*}\partial^n_{XY}((f^0+\mu)df^1 \ldots df^n):=df^0df^1 \ldots df^n\end{equation*}

\begin{definition}(see \cite[Definition 5.8]{BB1})
Let $(\mathcal S,\hat\partial)$ be a DG-semicategory. A closed graded trace of dimension $n$ on $\mathcal S$ is a collection of $\mathbb{C}$-linear maps $\hat T:=\{\hat T_X: Hom^n_\mathcal{S}(X,X) \longrightarrow \mathbb{C}\}_{X \in Ob(\mathcal{S})}$ satisfying the following two conditions
\begin{align}
&\hat T_X\left(\hat\partial^{n-1}(f)\right)=0 \label{closed}\\
&\hat T_X(gg')=(-1)^{ij}~ \hat T_Y(g'g)\label{trace}
\end{align}
for all $f \in Hom^{n-1}_{\mathcal S}(X,X)$, $g \in Hom^i_{\mathcal S}(Y,X)$, $g' \in Hom^j_{\mathcal S}(X,Y)$ and $i+j=n$.
\end{definition}

\begin{prop} \label{one-one}  (see \cite[Proposition 5.13]{BB1}) Let $\mathcal C$ be a small $\mathbb C$-category and let $T$ be a closed graded trace of dimension
$n$ on the  universal DG-semicategory $\Omega \mathcal{C}$. Let $CN^\bullet(\mathcal C)=\{Hom(CN_n(\mathcal C),\mathbb C)\}_{n\geq 0}$
be the dual of the cyclic nerve of $\mathcal C$. Define $\phi\in CN^n(\mathcal C)$ by setting
\begin{equation}\label{assco}
\phi(f^0 \otimes f^1 \otimes \ldots \otimes f^n):={T}_{X_0}(f^0df^1 \ldots df^n)
\end{equation}
for any $f^0 \otimes f^1 \otimes \ldots \otimes f^n  \in Hom_\mathcal{C}(X_1,X_0) \otimes Hom_\mathcal{C}(X_2,X_1) \otimes \ldots \otimes Hom_\mathcal{C} (X_0,X_n)$. Then,
$\phi$ is a cyclic cocycle over $\mathcal C$ and the association in \eqref{assco} determines a one-one correspondence between $n$-cyclic cocycles over $\mathcal{C}$ and $n$-dimensional closed graded traces on $\Omega \mathcal{C}$.
\end{prop}

\begin{definition}\label{defsrc} (see \cite[Definition 5.10]{BB1})
 Let $\mathcal{C}$ be a small $\mathbb{C}$-category. An $n$-dimensional cycle $(\mathcal{S},\hat{\partial},\hat{T},\rho)$ over $\mathcal C$ consists of the following data:
\begin{itemize}
\item[(i)] A DG-semicategory $(\mathcal S,\hat\partial)$ where $\mathcal S^0$ is an ordinary category. 
\item[(ii)]  An $n$-dimensional closed graded trace $\hat{T}$  on $\mathcal{S}$.
\item[(iii)] A $\mathbb C$-linear semifunctor $\rho:\mathcal C \longrightarrow \mathcal{S}^0$.
\end{itemize} The character $\phi\in CN^n(\mathcal C)$ of $(\mathcal{S},\hat{\partial},\hat{T},\rho)$  is defined by setting
\begin{equation*}
\phi(f^0 \otimes f^1 \otimes \ldots \otimes f^n)={T}_{X_0}(\rho(f^0)\hat\partial(\rho(f^1)) \ldots \hat\partial (\rho(f^n)))
\end{equation*} for an element  $f^0 \otimes f^1 \otimes \ldots \otimes f^n  \in Hom_\mathcal{C}(X_1,X_0) \otimes Hom_\mathcal{C}(X_2,X_1) \otimes \ldots \otimes Hom_\mathcal{C} (X_0,X_n)$ in $CN_n(\mathcal C)$.

\end{definition}

Let $(\mathscr{H},\mathcal F)$ be a pair that satisfies conditions (1) and (2) in Definition \ref{dD3.1}. We define a graded-semicategory
$\Omega'\mathcal C=\Omega_{(\mathscr H,\mathcal F)}\mathcal C$ as follows: we put $Ob(\Omega'\mathcal C):=Ob(\mathcal C)$ and for any $X$, $Y\in \mathcal C$, $j\geq 0$, we set
$Hom^j_{\Omega'{\mathcal C}}(X,Y)$ to be the linear span in $\mathcal B(\mathscr H(X),\mathscr H(Y))$ of the operators
\begin{equation}
\mathscr H(\tilde{f}^0)[\mathcal F,f^1][\mathcal F,f^2] \ldots [\mathcal F,f^{j}]
\end{equation} where  $\tilde{f}^0\otimes f^1\otimes ...\otimes f^j$ is a homogeneous element of degree $j$ in $Hom_{\Omega\mathcal{C}}(X,Y)$. Here, we write
$\mathscr H(\tilde{f}^0)=\mathscr H({f}^0)+\mu \cdot id$, where $\tilde{f}^0=f^0+\mu$.  Using the fact that
\begin{equation*}
[\mathcal F,f]\mathscr H(f')=[\mathcal F,f\circ f']-\mathscr H(f)[\mathcal F,f']
\end{equation*} for composable morphisms $f$, $f'$ in $\mathcal C$, we observe that $\Omega'\mathcal C$ is closed under composition.
We set 
\begin{equation*}
\begin{array}{c}
\partial':=[\mathcal{F},-]:   \mathcal{B}\left(\mathscr{H}(X),\mathscr{H}(Y)\right) \longrightarrow \mathcal{B}\left(\mathscr{H}(X),\mathscr{H}(Y)\right)\\
\partial' T=[\mathcal{F},T]=\mathcal{F}_Y \circ T -(-1)^{|T|} T\circ \mathcal{F}_X\\
\end{array}
\end{equation*} We now have the following Lemma. 
\begin{comment}
For each $X,Y \in Ob(\Omega'\mathcal{C})$,  the differential $\partial'^j_{XY}:Hom^j_{\Omega'\mathcal{C}}(X,Y) \longrightarrow Hom^{j+1}_{\Omega'\mathcal{C}}(X,Y)$ is determined by setting
\begin{equation}\label{part'}
\begin{array}{ll}
\partial'^j_{XY}(\mathscr H(\tilde{f}^0)[\mathcal F,f^1][\mathcal F,f^2] \ldots [\mathcal F,f^{j}])&=[\mathcal F,\textrm{ }\mathscr H(\tilde{f}^0)[\mathcal F,f^1][\mathcal F,f^2] \ldots [\mathcal F,f^{j}]]\\
&=\mathcal F_Y\circ \left(\mathscr H(\tilde{f}^0)[\mathcal F,f^1][\mathcal F,f^2] \ldots [\mathcal F,f^{j}]\right)\\ &\textrm{ }-(-1)^j \left(\mathscr H(\tilde{f}^0)[\mathcal F,f^1][\mathcal F,f^2] \ldots [\mathcal F,f^{j}]\right)\circ\mathcal F_X\\
\end{array}
\end{equation}
\end{comment}
%%%%%%%%%%%%%%%%%%%%%%%%

\begin{lemma} \label{srLe1} Let $(\mathscr{H},\mathcal F)$ be a pair that satisfies conditions (1) and (2) in Definition \ref{dD3.1}. Then, 

\smallskip
(a) $(\Omega'\mathcal C,\partial')$  is a DG-semicategory and $\Omega'^0\mathcal C$ is an ordinary category. 

\smallskip
(b) There is a canonical semifunctor $\rho'=\rho_{\mathscr H}:\mathcal C\longrightarrow \Omega'^0\mathcal C$ which is identity on objects and takes any $f\in Hom_{\mathcal C}(X,Y)$
to $\mathscr H(f)\in \mathcal B(\mathscr H(X),\mathscr H(Y))$. This extends to  a unique DG-semifunctor  $\hat{\rho}'=\hat{\rho}_{\mathscr H}:(\Omega\mathcal C,\partial)\longrightarrow (\Omega'\mathcal C,\partial')$  such that the restriction of  $\hat{\rho}'$  to $\mathcal C$ is identical to $\rho'$.

\smallskip
(c) Suppose that $(\mathscr{H},\mathcal F)$ is  a $p$-summable  even  Fredholm module. Choose $n\geq p-1$. Then, for $X$, $Y\in Ob(\mathcal C)$ and $k\geq 0$, we have $Hom^k_{\Omega'{\mathcal C}}(X,Y)
\subseteq \mathcal B^{(n+1)/k}(\mathscr H(X),\mathscr H(Y))$. 
\end{lemma}

\begin{proof}Since each $\mathcal{F}_X$ is a degree 1 operator and $\mathcal F_Y[\mathcal F,f]=-[\mathcal F,f]\mathcal F_X$ for any 
$f\in Hom_{\mathcal C}(X,Y)$, we have $\partial'\left(Hom^j_{\Omega' \mathcal{C}}(X,Y)\right) \subseteq Hom^{j+1}_{\Omega' \mathcal{C}}(X,Y)$.
We now check that $\partial'^2=0$. For any homogeneous element $\mathscr H(\tilde{f}^0)[\mathcal F,f^1][\mathcal F,f^2] \ldots [\mathcal F,f^{j}]$ of degree $j$, we have
\begin{equation*}
\begin{array}{ll}
\partial'^2\left(\mathscr H(\tilde{f}^0)[\mathcal F,f^1][\mathcal F,f^2] \ldots [\mathcal F,f^{j}]\right)&=\partial'\left(\mathcal F_Y\circ \left(\mathscr H(\tilde{f}^0)[\mathcal F,f^1][\mathcal F,f^2] \ldots [\mathcal F,f^{j}]\right)\right)\\ 
&\textrm{ }-(-1)^j \partial' \left(\left(\mathscr H(\tilde{f}^0)[\mathcal F,f^1][\mathcal F,f^2] \ldots [\mathcal F,f^{j}]\right)\circ\mathcal F_X\right)\\
& =\mathcal F_Y^2 \circ \mathscr H(\tilde{f}^0)[\mathcal F,f^1][\mathcal F,f^2] \ldots [\mathcal F,f^{j}]\\
& \quad -(-1)^{j+1} \mathcal{F}_Y \circ \mathscr H(\tilde{f}^0)[\mathcal F,f^1][\mathcal F,f^2] \ldots [\mathcal F,f^{j}]\circ\mathcal F_X \\
& \quad - (-1)^j \big( \mathcal{F}_Y \circ \mathscr H(\tilde{f}^0)[\mathcal F,f^1][\mathcal F,f^2] \ldots [\mathcal F,f^{j}]\circ\mathcal F_X \\
&\quad - (-1)^{j+1}  \mathscr H(\tilde{f}^0)[\mathcal F,f^1][\mathcal F,f^2] \ldots [\mathcal F,f^{j}]\circ\mathcal F_X^2      \big)\\
& =0
\end{array}
\end{equation*}
The fact that $\partial'$ is compatible with composition follows exactly as in the proof of \cite[Proposition 5.5 (i)]{BB1}. It is also easy to see that $\Omega'^0\mathcal C$ is an ordinary category. 

\smallskip
(b) This is immediate using the universal property in \cite[Proposition 5.5]{BB1}.

\smallskip
(c) This is a consequence of the H\"older's inequality and the condition \eqref{stneve} in Definition \ref{dD3.1}.
\end{proof}

For any $\mathbb Z_2$-graded Hilbert space $\mathcal H$, the grading operator on it will be denoted by $\epsilon_{\mathcal H}$ or simply $\epsilon$. 
For any $T \in \mathcal{B}(\mathscr H(X), \mathscr H(Y))$ such that $[\mathcal{F},T] \in \mathcal{B}^1(\mathscr H(X), \mathscr H(Y))$, we define
\begin{equation*}
Tr_s(T):=\frac{1}{2}~Tr\left(\epsilon \mathcal F_Y[\mathcal F,T]\right)=\frac{1}{2}~Tr\left(\epsilon \mathcal F_Y\partial'(T)\right)=\frac{1}{2}~ Tr\left(\epsilon\mathcal{F}_Y(\mathcal{F}_Y \circ T - (-1)^{|T|}~ T \circ \mathcal{F}_X)\right)
\end{equation*}

\begin{prop}\label{srLe2} Let $(\mathscr{H},\mathcal F)$ be  a $p$-summable  even Fredholm module over $\mathcal C$. Take $ 2m \geq p-1$. Then, the collection
\begin{equation}
\hat{Tr}_s=\{Tr_s:Hom^{2m}_{\Omega'\mathcal{C}}(X,X) \longrightarrow \mathbb C\}_{X\in Ob(\mathcal C)}
\end{equation} defines a closed graded trace of dimension $2m$ on $(\Omega'\mathcal C,\partial')$. 

\end{prop}

\begin{proof} From Lemma \ref{srLe1}(a), it is clear that for any $T\in Hom^{2m}_{\Omega'\mathcal{C}}(X,X)$, we have $[\mathcal F,T]\in Hom^{2m+1}_{\Omega'\mathcal{C}}(X,X)$. Applying Lemma \ref{srLe1}(c), it follows that $[\mathcal F,T]\in  \mathcal B^1(\mathscr H(X),\mathscr H(X))$. Accordingly, each of the maps $Tr_s:Hom^{2m}_{\Omega'\mathcal{C}}(X,X) \longrightarrow \mathbb C$ is well-defined.

\smallskip
For $T'\in Hom^{2m-1}_{\Omega'\mathcal{C}}(X,X)$, we notice that
\begin{equation*}
Tr_s(\partial'T')=\frac{1}{2}~Tr\left(\epsilon \mathcal F_X(\partial'^2T')\right)=0
\end{equation*} We now consider $T_1\in Hom^{i}_{\Omega'\mathcal{C}}(X,Y)$, $T_2\in Hom^{j}_{\Omega'\mathcal{C}}(Y,X)$ such that $i+j=2m$. We notice that
\begin{equation}\label{sre3.4}
\begin{array}{c}
\epsilon \mathcal F_Y\partial'(T_1)=\partial'(T_1)\epsilon \mathcal F_X \qquad \epsilon \mathcal F_X\partial'(T_2)=\partial'(T_2)\epsilon \mathcal F_Y
\end{array}
\end{equation} We note that $i\equiv j \mbox{(mod $2$)}$. Using \eqref{sre3.4} and \eqref{commutr}, we now have
\begin{equation*}
\begin{array}{ll}
2\cdot Tr_s(T_1T_2)=Tr\left(\epsilon \mathcal F_Y\partial'(T_1T_2)\right) & = Tr\left(\epsilon \mathcal F_Y\partial'(T_1)T_2\right)+(-1)^iTr\left(\epsilon \mathcal F_YT_1\partial'(T_2)\right)\\
& = Tr\left(\partial'(T_1)\epsilon \mathcal F_XT_2\right)+(-1)^iTr\left(\partial'(T_2)\epsilon \mathcal F_YT_1\right)\\
& =Tr\left(\epsilon \mathcal F_XT_2\partial'(T_1)\right)+(-1)^iTr\left(\epsilon \mathcal F_X\partial'(T_2)T_1\right)\\
& =  Tr\left(\epsilon \mathcal F_XT_2\partial'(T_1)\right)+(-1)^jTr\left(\epsilon \mathcal F_X\partial'(T_2)T_1\right)\\
&= (-1)^{ij}2\cdot Tr_s(T_2T_1)
\end{array}
\end{equation*}
\end{proof}

\begin{theorem}\label{evencyc}
Let $(\mathscr{H},\mathcal F)$ be  a $p$-summable  even Fredholm module over $\mathcal C$. Take $2m \geq p-1$.  Then, the tuple
$(\Omega'\mathcal C,\partial',\hat{Tr}_s,\rho')$ defines a $2m$-dimensional cycle over $\mathcal C$. Then, $\phi^{2m} \in CN^{2m}(\mathcal{C})$ defined by
\begin{equation*}
\phi^{2m}(f^0 \otimes f^1 \otimes \ldots \otimes f^{2m}):= Tr_s\left(\mathscr H({f}^0)[\mathcal F,f^1][\mathcal F,f^2] \ldots [\mathcal F,f^{2m}]\right)
\end{equation*}
for any $f^0 \otimes f^1 \otimes \ldots \otimes f^{2m} \in Hom_{\mathcal{C}}(X_1,X) \otimes Hom_{\mathcal{C}}(X_2,X_1) \otimes \ldots \otimes Hom_{\mathcal{C}}(X,X_{2m})$ is a cyclic cocycle over $\mathcal C$.

\end{theorem}

\begin{proof} It  follows directly from Lemma \ref{srLe1} and Proposition \ref{srLe2} that  $(\Omega'\mathcal C,\partial',\hat{Tr}_s,\rho')$ is a $2m$-dimensional cycle over $\mathcal C$. 
The rest follows from the definition of the character of a cycle in Definition \ref{defsrc} and the result of \cite[Theorem 5.11]{BB1}. 
\end{proof}

We will refer to $\phi^{2m}$ as the $2m$-dimensional character associated with the $p$-summable  even Fredholm module $(\mathscr{H},\mathcal F)$ over the category $\mathcal{C}$.

\begin{remark} 
The appearance of only even cyclic cocycles in Theorem \ref{evencyc} is due to the following fact from \cite[Lemma 2 a)]{C2}: if $T \in \mathcal{B}(\mathscr H(X), \mathscr H(X))$ is homogeneous of odd degree, then $Tr_s(T)=0$.
\end{remark}

\section{Periodicity of Chern character for even Fredholm modules}
Let $\mathcal{C}$ be a small $\mathbb{C}$-category and let $(\mathscr{H},\mathcal{F})$ be a $p$-summable even Fredholm module over $\mathcal{C}$. We take  $2m \geq p-1$. Let $\phi^{2m}$ be the $2m$-dimensional character associated to the Fredholm module $(\mathscr{H},\mathcal F)$.

\smallskip Since $\mathcal{B}^p(\mathscr H(X), \mathscr H(Y)) \subseteq \mathcal{B}^q(\mathscr H(X), \mathscr H(Y))$ for any $p \leq q$, the Fredholm module $(\mathscr{H},\mathcal{F})$ is also $(p+2)$-summable. Using Theorem \ref{evencyc}, we then have the $(2m+2)$-dimensional character $\phi^{2m+2}$ associated to $(\mathscr{H},\mathcal{F})$.
We will show that the cyclic cocycles $\phi^{2m}$ and $\phi^{2m+2}$ are related to each other via the periodicity operator. 

\smallskip
For small $\mathbb{C}$-categories $\mathcal C$ and $\mathcal C'$,  we recall from \cite[Theorem 7.11]{BB1} that there is a pairing on cyclic cocycles
\begin{equation}\label{pairty}
Z^r_\lambda(\mathcal{C}) \otimes Z^s_\lambda(\mathcal{C}') \longrightarrow Z^{r+s}_\lambda(\mathcal{C} \otimes \mathcal{C}') \qquad \phi \otimes \phi' \mapsto \phi \# \phi'
\end{equation} which descends to a pairing on cyclic cohomologies:
\begin{equation}\label{pair}
H^r_\lambda(\mathcal{C}) \otimes H^s_\lambda(\mathcal{C}') \longrightarrow H^{r+s}_\lambda(\mathcal{C} \otimes \mathcal{C}') 
\end{equation}
given by
\begin{equation*}
{(\hat{T}^\phi \# \hat{T}^{\phi'})}_{(X,X')}(f \otimes f'):=\hat{T}^\phi_X( f_r)\hat{T}^{\phi'}_{X'}(f'_s)
\end{equation*}
for any $f \otimes f'= \underset{i+j=r+s}{\sum}(f_i \otimes f'_j)\in Hom^{r+s}_{\mathcal S \otimes  \mathcal S'}\left((X,X'),(X,X')\right)$. Here $\phi$ and $\phi'$ are expressed respectively as the characters of $r$ and $s$-dimensional cycles $(\mathcal{S},\hat\partial, \hat{T}^\phi,\rho)$ and $(\mathcal{S}',\hat\partial',\hat{T}^{\phi'},\rho')$ over $\mathcal{C}$ and $\mathcal{C}'$. In particular, $\phi \# \phi'$ is the character of  the $(r+s)$-dimensional cycle $\left(\mathcal{S} \otimes \mathcal{S}', \hat{\partial} \otimes \hat{\partial'}, \hat{T}^\phi\# \hat{T}^{\phi'}, \rho \otimes \rho' \right)$ over $\mathcal{C} \otimes \mathcal{C}'$. For a morphism $f$ in $\mathcal C$, we will often suppress the functor $\rho$ and write
the morphism $\rho(f)$ in $\mathcal S^0$ simply as $f$.  Similarly, when there is no danger of confusion, we will often write the morphism $\mathscr H(f)$ simply as $f$.

\smallskip
Now setting $\mathcal{C}'=\mathbb{C}$ (the category with one object) and considering the cyclic cocycle $\psi \in H^2_\lambda(\mathbb{C})$ determined by $\psi(1,1,1)=1$, we obtain the periodicity operator:
\begin{equation*} 
S:Z^r_\lambda(\mathcal{C}) \longrightarrow  Z^{r+2}_\lambda(\mathcal{C}) \qquad S(\phi):=\phi \# \psi 
\end{equation*}
for any $r \geq 0$ and $\phi \in Z^r_\lambda(\mathcal{C})$. 

\begin{lemma}\label{Sform} Let $\phi \in Z^r_\lambda(\mathcal{C})$.
For any $f^0 \otimes f^1 \otimes \ldots \otimes f^{r+2} \in CN_{r+2}(\mathcal{C})$, we have

\smallskip
$\begin{array}{ll}
(S(\phi))(f^0 \otimes f^1 \otimes \ldots \otimes f^{r+2})
&=\hat{T}^\phi_X(f^0f^1f^2\hat\partial f^3 \ldots \hat\partial f^{r+2})+\hat{T}^\phi_X(f^0\hat\partial f^1(f^2f^3) \ldots \hat\partial f^{r+2})+ \ldots \\
&\quad +\hat{T}^\phi_X(f^0\hat\partial f^1 \ldots \hat\partial f^{i-1} (f^if^{i+1})\hat\partial f^{i+2} \ldots \hat\partial f^{r+2}) + \ldots \\
& \quad +\hat{T}^\phi_X(f^0\hat\partial f^1 \ldots \hat\partial f^r (f^{r+1}f^{r+2}))
\end{array}$
\end{lemma}

\begin{proof} We consider the $2$-dimensional trace $\hat{T}^\psi$ on the DG-semicategory $(\Omega\mathbb C,\partial )$ such that $\psi\in Z^2_\lambda(\mathbb C)$ is the character of the corresponding cycle over $\mathbb C$.  We first observe that we have the following equalities in $\Omega \mathbb{C}$:
\begin{equation*}
\partial 1= (\partial1)1+1 (\partial 1), \qquad 1(\partial1)1=0, \qquad 1(\partial 1)^2=(\partial1)^21
\end{equation*}

We illustrate the proof for $r=2$. The general case will follow similarly.  By definition, we have

\smallskip
$\begin{array}{ll}
&(S(\phi))(f^0 \otimes f^1 \otimes f^2 \otimes f^3 \otimes f^{4})\\
& \quad =(\phi \# \psi)(f^0 \otimes f^1 \otimes f^2 \otimes f^3 \otimes f^{4})\\
& \quad ={(\hat{T}^\phi \# \hat{T}^\psi)}\left( (f^0 \otimes 1)(\hat{\partial} \otimes \partial)(f^1 \otimes 1) (\hat{\partial} \otimes \partial)(f^2 \otimes 1)   (\hat{\partial} \otimes \partial)(f^3 \otimes 1) (\hat{\partial} \otimes \partial)(f^{4} \otimes 1) \right)\\
& \quad ={(\hat{T}^\phi \# \hat{T}^\psi)}\left( (f^0 \otimes 1) (\hat{\partial}f^1 \otimes 1 +f^1 \otimes \partial 1) (\hat \partial f^2 \otimes 1 +f^2 \otimes \partial1)(\hat \partial f^3 \otimes 1 +f^3 \otimes \partial1) (\hat \partial f^{4} \otimes 1 +f^{4} \otimes \partial1) \right)\\
& \quad ={(\hat{T}^\phi \# \hat{T}^\psi)} \Big(f^0 \hat \partial f^1 \hat \partial f^2 \hat \partial f^3 \hat \partial f^4 \otimes 1+f^0 \hat \partial f^1 \hat \partial f^2 \hat\partial f^3f^4  \otimes 1 \partial 1 + f^0\hat \partial f^1 \hat \partial f^2f^3f^4 \otimes 1(\partial 1)^2 +\\
&\qquad f^0 \hat \partial f^1f^2f^3 \hat \partial f^4 \otimes 1(\partial 1)^21 + f^0 \hat \partial f^1f^2f^3f^4 \otimes 1(\partial 1)^3 + f^0f^1f^2 \hat \partial f^3\hat \partial f^4 \otimes 1 (\partial 1)^2 \\
&\quad + f^0f^1f^2\hat \partial f^3 f^4 \otimes 1(\partial 1)^3 -
 f^0f^1f^2f^3\hat \partial f^4 \otimes 1(\partial 1)^31 + f^0f^1f^2f^3f^4 \otimes 1(\partial 1)^4\Big)\\
& \quad =\hat{T}^\phi\left(f^0 \hat \partial f^1 \hat \partial f^2f^3f^4\right) \hat{T}^\psi\left(1(\partial 1)^2\right) + \hat{T}^\phi\left(f^0\hat \partial f^1f^2f^3\hat \partial f^4\right) \hat{T}^\psi\left(1(\partial 1)^21\right)\\
& \qquad + \hat{T}^\phi\left(f^0f^1f^2\hat \partial f^3 \hat \partial f^4\right) \hat{T}^\psi\left(1(\partial 1)^2\right)\\
& \quad= \hat{T}^\phi\left(f^0\hat \partial f^1 \hat \partial f^2f^3f^4\right) + \hat{T}^\phi\left(f^0 \hat \partial f^1f^2f^3\hat\partial f^4\right)+ \hat{T}^\phi\left(f^0f^1f^2\hat \partial f^3 \hat \partial f^4\right) 
\end{array}$

\smallskip
The last equality follows by using the fact that $\hat{T}^\psi\left(1(\partial 1)^2\right)=\psi(1,1,1)=1$.
\end{proof}

\begin{prop}\label{Sbound}
Let  $\phi$ be the character of an $r$-dimensional cycle $(\mathcal S,\hat{\partial},\hat{T}^{\phi},\rho)$ over $\mathcal{C}$. Then, $S(\phi)$ is a coboundary.  In particular, we have $S(\phi)=b\psi$, where $\psi \in CN^{r+1}(\mathcal{C})$ is given by
\begin{equation*}
\psi(f^0 \otimes f^1 \otimes \ldots \otimes f^{r+1})=\sum\limits_{j=1}^{r+1} (-1)^{j-1}~\hat{T}^{\phi}\left(f^0\hat{\partial}f^1\ldots  \hat{\partial}f^{j-1} f^j \hat{\partial}f^{j+1} \ldots  \hat{\partial}f^{r+1}\right)
\end{equation*}
\end{prop}
\begin{proof}
Again, we illustrate the case of $r=2$. The general computation is similar.
\begin{equation*}
\begin{array}{ll}
&(b\psi)(f^0\otimes f^1 \otimes f^2 \otimes f^3 \otimes f^4)\\
&=\psi(f^0f^1 \otimes f^2 \otimes f^3 \otimes f^4) - \psi(f^0 \otimes f^1f^2 \otimes f^3 \otimes f^4) + \psi(f^0 \otimes f^1 \otimes f^2f^3 \otimes f^4) - \psi(f^0 \otimes f^1 \otimes f^2 \otimes f^3f^4)\\
&\quad + \psi(f^4f^0 \otimes f^1 \otimes f^2 \otimes f^3)\\
&=\hat{T}^\phi(f^0f^1f^2\hat{\partial}f^3\hat{\partial}f^4)-\hat{T}^\phi(f^0f^1\hat{\partial}f^2f^3\hat{\partial}f^4)+ \hat{T}^\phi(f^0f^1\hat{\partial}f^2\hat{\partial}f^3f^4) \\
&\quad  -\hat{T}^\phi(f^0f^1f^2\hat{\partial}f^3\hat{\partial}f^4)+\hat{T}^\phi(f^0\hat{\partial}(f^1f^2)f^3\hat{\partial} f^4-\hat{T}^\phi(f^0\hat{\partial}(f^1f^2)\hat{\partial}f^3f^4)\\
&\quad +\hat{T}^\phi(f^0f^1\hat{\partial}(f^2f^3)\hat{\partial}f^4)-\hat{T}^\phi(f^0\hat{\partial}f^1f^2f^3\hat{\partial}f^4)+ \hat{T}^\phi(f^0\hat{\partial}f^1\hat{\partial}(f^2f^3)f^4)\\
&\quad -\hat{T}^\phi(f^0f^1\hat{\partial}f^2\hat{\partial}(f^3f^4))+ \hat{T}^\phi(f^0\hat{\partial}f^1f^2\hat{\partial}(f^3f^4))- \hat{T}^\phi(f^0\hat{\partial}f^1\hat{\partial}f^2f^3f^4)\\
&\quad +\hat{T}^\phi(f^4f^0f^1\hat{\partial}f^2\hat{\partial}f^3)- \hat{T}^\phi(f^4f^0\hat{\partial}f^1f^2\hat{\partial}f^3) + \hat{T}^\phi(f^4f^0\hat{\partial}f^1\hat{\partial}f^2f^3)\\
&= \hat{T}^\phi(f^0f^1f^2\hat{\partial}f^3\hat{\partial}f^4) +\hat{T}^\phi(f^0\hat{\partial}f^1(f^2f^3)\hat{\partial}f^4) + \hat{T}^\phi(f^0\hat{\partial}f^1\hat{\partial}f^2f^3f^4) \\
&= (S(\phi))(f^0\otimes f^1 \otimes f^2 \otimes f^3 \otimes f^4)
\end{array}
\end{equation*}
\end{proof}

\begin{prop}
Let $\mathcal{C}$ be a small $\mathbb{C}$-category and let $(\mathscr{H},\mathcal{F})$ be a $p$-summable even Fredholm module over $\mathcal{C}$. Take $ 2m \geq p-1$. Then, 
\begin{equation*}
S(\phi^{2m})=-(m+1)\phi^{2m+2} \qquad \text{in}~ H^{2m+2}_\lambda(\mathcal{C})
\end{equation*}
\end{prop}
\begin{proof}
We will show that $S(\phi^{2m})+(m+1)\phi^{2m+2}=b\psi$ for some $\psi \in Z^{2m+1}_\lambda(\mathcal{C})$.
By Theorem \ref{evencyc}, we know that $\phi^{2m}$ is the character of the $2m$-dimensional cycle $(\Omega' \mathcal{C}, \partial', \hat{Tr}_s,\rho')$ over the category $\mathcal{C}$. Applying Lemma \ref{Sform} and using the fact that $Tr_s(T)=0$ for any homogeneous $T$ of odd degree, we have
\begin{equation*}
\begin{array}{ll}
(S(\phi^{2m}))(f^0 \otimes f^1 \otimes \ldots \otimes f^{2m+2}) &= \sum\limits_{j=0}^{2m+1}  Tr_s\left(f^0[\mathcal F,f^1]\ldots [\mathcal F,f^{j-1}](f^jf^{j+1})[\mathcal F,f^{j+2}]\ldots [\mathcal F,f^{2m+2}] \right)
\end{array}
\end{equation*}
Further,
\begin{equation*}
\begin{array}{ll}
\phi^{2m+2}(f^0 \otimes f^1 \otimes \ldots \otimes f^{2m+2})=  Tr_s\left( f^0[\mathcal F,f^1]\ldots\ldots [\mathcal F,f^{2m+2}] \right)
\end{array}
\end{equation*}
so that
\begin{equation}\label{exp}
\begin{array}{lr}
&\left(S(\phi^{2m})+(m+1)\phi^{2m+2}\right)(f^0 \otimes f^1 \otimes \ldots \otimes f^{2m+2})\qquad \qquad \qquad \qquad \qquad \qquad \\
&= \sum\limits_{j=0}^{2m+1}  Tr_s\left( f^0[\mathcal F,f^1]\ldots [\mathcal F,f^{j-1}](f^jf^{j+1})[\mathcal F,f^{j+2}]\ldots [\mathcal F,f^{2m+2}] \right)\\
&  +(m+1)Tr_s\left(f^0[\mathcal F,f^1]\ldots\ldots [\mathcal F,f^{2m+2}] \right)\qquad \qquad \qquad \qquad \qquad
\end{array}
\end{equation}
We now consider $\psi=\sum\limits_{j=0}^{2m+1} (-1)^{j-1} \psi^j$, where
\begin{equation}\label{defbound}
\psi^j(f^0 \otimes f^1 \otimes \ldots \otimes f^{2m+1})=Tr\left(\epsilon \mathcal  F f^j [\mathcal F,f^{j+1}]\ldots [\mathcal F,f^{2m+1}][\mathcal F,f^0][\mathcal F,f^1]\ldots[\mathcal F,f^{j-1}] \right)
\end{equation}
Since $2m \geq p-1$ and $(\mathscr H,\mathcal F)$ is a $p$-summable even Fredholm module over $\mathcal C$,  it follows that the operator 
$\epsilon \mathcal  F f^j [\mathcal F,f^{j+1}]\ldots [\mathcal F,f^{2m+1}][\mathcal F,f^0][\mathcal F,f^1]\ldots[\mathcal F,f^{j-1}]$ is trace class. 

\smallskip
We observe that $\tau\psi^j=\psi^{j-1}$ for $1\leq j\leq 2m+1$ and $\tau \psi^0=\psi^{2m+1}$. It follows that  $(1-\lambda)(\psi)=0$. Hence, $\psi \in C^{2m+1}_\lambda(\mathcal{C})=Ker(1-\lambda)$. Using \eqref{defbound}, we have
\begin{equation*}
\begin{array}{ll}
&(b\psi^j)(f^0 \otimes f^1 \otimes \ldots \otimes f^{2m+2})\\
& \quad =\sum\limits_{i=0}^{2m+1} (-1)^i ~\psi^j(f^0 \otimes \ldots \otimes f^if^{i+1} \otimes \ldots \otimes f^{2m+2})+  \psi^j(f^{2m+2}f^0 \otimes f^1 \otimes \ldots \otimes f^{2m+1})\\
& \quad =Tr\left(\epsilon \mathcal Ff^{j+1}[\mathcal F,f^{j+2}]\ldots[\mathcal F,f^{2m+2}]f^0[\mathcal F,f^1]\ldots[\mathcal F,f^j]\right)~+\\& \qquad (-1)^{j-1}Tr\left(\epsilon \mathcal  Ff^{j+1}[\mathcal F,f^{j+2}]\ldots[\mathcal F,f^{2m+2}][\mathcal F,f^0][\mathcal F,f^1]\ldots f^j \right)+\\
& \qquad  Tr\left(\epsilon \mathcal Ff^{j}[\mathcal F,f^{j+1}]\ldots[\mathcal F,f^{2m+2}]f^0[\mathcal F,f^1]\ldots[\mathcal F,f^{j-1}]\right)
\end{array}
\end{equation*}
We now set $\beta^j=[\mathcal F,f^{j+2}]\ldots[\mathcal F,f^{2m+2}]f^0[\mathcal F,f^1]\ldots[\mathcal F,f^{j-1}]$.  Then,  we have
\begin{equation*}
\begin{array}{ll}
[\mathcal F,\beta^j]=\mathcal F\beta^j- (-1)^{2m}\beta^j \mathcal F&=\mathcal F[\mathcal F,f^{j+2}]\ldots[\mathcal F,f^{2m+2}]f^0[\mathcal F,f^1]\ldots[\mathcal F,f^{j-1}]-\\
&\qquad [\mathcal F,f^{j+2}]\ldots[\mathcal F,f^{2m+2}]f^0[\mathcal F,f^1]\ldots[\mathcal F,f^{j-1}]\mathcal F\\
&= (-1)^{j-1}[\mathcal F,f^{j+2}]\ldots [\mathcal F,f^{2m+2}][\mathcal F,f^0][\mathcal F,f^1]\ldots[\mathcal F,f^{j-1}]
\end{array}
\end{equation*}
With $\alpha^j=f^j \mathcal Ff^{j+1}$, we get
\begin{equation}
\begin{array}{ll}
&(-1)^{j-1}Tr\left(\epsilon \mathcal Ff^{j+1}[\mathcal F,f^{j+2}]\ldots[\mathcal F,f^{2m+2}][\mathcal F,f^0][\mathcal F,f^1]\ldots [\mathcal F,f^{j-1}]f^j \right)\\
& \quad =Tr\left(\epsilon \mathcal Ff^{j+1}[\mathcal F,\beta^j]f^j\right)=Tr_s\left(\alpha^j[\mathcal F,\beta^j]\right)=Tr_s([\mathcal F,\alpha^j]\beta^j)
\end{array}
\end{equation} where we have used the fact that $Tr_s$ is a closed graded trace and $Tr_s(T)=Tr(\epsilon T)$ for any operator that is trace class (see \cite[Lemma 2]{C2}).
Thus, we have
\begin{equation*}
\begin{array}{ll}
&(b\psi^j)(f^0 \otimes f^1 \otimes \ldots \otimes f^{2m+2})=-Tr_s\left([\mathcal F,f^j] \mathcal  Ff^{j+1}\beta^j\right)+Tr_s([\mathcal F,\alpha^j]\beta^j)+Tr_s\left( \mathcal Ff^j[\mathcal F,f^{j+1}]\beta^j\right)
\end{array}
\end{equation*}
Since
\begin{equation*}
\mathcal F[\mathcal F,f^jf^{j+1}]=\mathcal F[\mathcal F,f^j]f^{j+1}+\mathcal Ff^j[\mathcal F,f^{j+1}]=-[\mathcal F,f^j]\mathcal Ff^{j+1}+ \mathcal Ff^j[\mathcal F,f^{j+1}],
\end{equation*}
we obtain 
\begin{equation*}
\begin{array}{ll}
(b\psi^j)(f^0 \otimes f^1 \otimes \ldots \otimes f^{2m+2})&=Tr_s\big(\left(\mathcal F[\mathcal F,f^jf^{j+1}] + [\mathcal F,\alpha^j]\right)\beta^j\Big)
\end{array}
\end{equation*}

As 
\begin{equation*}
\mathcal F[\mathcal F,f^jf^{j+1}] +[\mathcal F,\alpha^j]=\mathcal F[\mathcal F,f^jf^{j+1}] +\mathcal F \alpha^j + \alpha^j \mathcal F =[\mathcal F,f^j][\mathcal F,f^{j+1}]+2f^jf^{j+1}
\end{equation*}
 we get
\begin{equation*}
\begin{array}{ll}
(b\psi)(f^0 \otimes f^1 \otimes \ldots \otimes f^{2m+2})&=\sum\limits_{j=0}^{2m+1} (-1)^{j-1} (b\psi^j) (f^0 \otimes f^1 \otimes \ldots \otimes f^{2m+2})\\
&=\sum\limits_{j=0}^{2m+1}(-1)^{j-1} \big( 2~Tr_s\left(f^jf^{j+1}\beta^j\right) + Tr_s\left([\mathcal F,f^j][\mathcal F,f^{j+1}]\beta^j\right) \big)\\
&=\sum\limits_{j=0}^{2m+1} 2~Tr_s \left(f^0[\mathcal F,f^1]\ldots [\mathcal F,f^{j-1}](f^jf^{j+1})[\mathcal F,f^{j+2}]\ldots [\mathcal F,f^{2m+2}] \right)\\
& \quad +\sum\limits_{j=0}^{2m+1}Tr_s\left(f^0[\mathcal F,f^1] \ldots  [\mathcal F,f^{2m+2}]\right)\\
&=\sum\limits_{j=0}^{2m+1}~2~Tr_s\left(f^0[\mathcal F,f^1]\ldots [\mathcal F,f^{j-1}](f^jf^{j+1})[\mathcal F,f^{j+2}]\ldots [\mathcal F,f^{2m+2}] \right)\\
& \quad +(2m+2)Tr\left(f^0[\mathcal F,f^1] \ldots  [\mathcal F,f^{2m+2}]\right)
\end{array}
\end{equation*}
The result now follows by \eqref{exp}.
\end{proof}

\section{Homotopy invariance of the Chern character}

Let $SHilb_2$ be the category whose objects are separable Hilbert spaces and for any two objects $\mathcal H,\mathcal H' \in SHilb_2$,  the morphism space is given by
\begin{equation*}
SHilb_2(\mathcal H,\mathcal H'):=\mathcal{B}(\mathcal H \oplus \mathcal  H, \mathcal H' \oplus\mathcal  H')
\end{equation*} For any object $\mathcal H\in SHilb_2$, the morphism in $SHilb_2(\mathcal H,\mathcal H)= \mathcal{B}(\mathcal H \oplus \mathcal H, \mathcal H \oplus\mathcal  H)$ given by the matrix 
$\begin{pmatrix} 0 & 1\\ 1 & 0 \\ \end{pmatrix}$ swapping the two copies of $\mathcal H$ will be denoted by $F(\mathcal H)$.   
We also consider the functor $i:SHilb_2 \longrightarrow SHilb_{\mathbb{Z}_2}$ given by
\begin{equation}\label{i}
i(\mathcal H):=\mathcal H \oplus \mathcal H \qquad i(T):=T
\end{equation}
for any $\mathcal H \in SHilb_2$ and $T\in SHilb_2(\mathcal H,\mathcal H')$.

\begin{lemma}\label{5.2x}
Let $\mathcal C$ be a small $\mathbb{C}$-category and $\{\mathscr{H}_t:\mathcal{C} \longrightarrow SHilb_2\}_{t \in [0,1]}$ be a family of functors such that
for each $X\in Ob(\mathcal C)$, we have $\mathscr H_t(X)=\mathscr H_{t'}(X)=\mathscr H(X)$ for all $t$, $t'\in [0,1]$.

\smallskip For $t \in [0,1]$, we set $\widehat{\mathscr{H}}_t:=i \circ \mathscr{H}_t$. For each $f:X \longrightarrow Y$ in $\mathcal{C}$, we assume that the function
\begin{equation*} p_f:[0,1] \longrightarrow SHilb_{\mathbb{Z}_2}(\widehat{\mathscr H}_t(X),\widehat{\mathscr H}_t(Y))=\mathcal B(\mathscr H(X)
\oplus \mathscr H(X),
\mathscr H(Y)\oplus \mathscr H(Y))\qquad t \mapsto \widehat{\mathscr{H}}_t
(f)\end{equation*} is strongly $C^1$. Then if $\delta_t(f):=p_f'(t)$, we have
\begin{equation*}
\delta_t(fg)=\widehat{\mathscr{H}}_t(f) \circ \delta_t(g) + \delta_t(f) \circ \widehat{\mathscr{H}}_t
(g)
\end{equation*} for composable morphisms $f$, $g$ in $\mathcal C$.
\end{lemma}
\begin{proof} We have
\begin{equation*}
\begin{array}{ll}
&\delta_t(fg)-\widehat{\mathscr{H}}_t(f) \circ \delta_t(g) - \delta_t(f) \circ \widehat{\mathscr{H}}_t
(g)\\
& \quad =p'_{fg}(t)-\widehat{\mathscr{H}}_t(f) \circ p'_g(t) -p'_f(t)  \circ \widehat{\mathscr{H}}_t
(g)\\
& \quad =\lim\limits_{s \to 0} \frac{1}{s} \left(p_{fg}(t+s)-p_{fg}(t) - \widehat{\mathscr{H}}_t(f)~ \circ ~p_g(t+s)+  \widehat{\mathscr{H}}_t(f)~ \circ~ p_g(t)- p_f(t+s)~ \circ~ \widehat{\mathscr{H}}_t
(g)+p_f(t)~ \circ~ \widehat{\mathscr{H}}_t(g)\right)\\
& \quad =\lim\limits_{s \to 0} \frac{1}{s}\left(\widehat{\mathscr{H}}_{t+s}(fg)   -\widehat{\mathscr{H}}_{t}(fg)  - \widehat{\mathscr{H}}_t(f)\widehat{\mathscr{H}}_{t+s}(g) +  \widehat{\mathscr{H}}_t(f) \widehat{\mathscr{H}}_{t}(g) - \widehat{\mathscr{H}}_{t+s}(f) \widehat{\mathscr{H}}_t
(g)+\widehat{\mathscr{H}}_{t}(f) \widehat{\mathscr{H}}_t(g)\right)\\
& \quad =\lim\limits_{s \to 0} \frac{1}{s}\left(\widehat{\mathscr{H}}_{t+s}(f)-\widehat{\mathscr{H}}_{t}(f)\right)\left(\widehat{\mathscr{H}}_{t+s}(g)-\widehat{\mathscr{H}}_{t}(g)\right)\\
& \quad =\lim\limits_{s \to 0}\frac{1}{s} \left(p_f(t+s)-p_f(t)\right) \left(p_g(t+s)-p_g(t)\right)\\
& \quad = p_f'(t) \lim\limits_{s \to 0} \left(p_g(t+s)-p_g(t)\right)=0 
\end{array}\qedhere
\end{equation*}
\end{proof}

For each $n \in \mathbb{Z}_{\geq 0}$, we define an operator $A:CN^{n}(\mathcal{C}) \longrightarrow CN^n(\mathcal{C})$ given by
\begin{equation*}
A:=1+\lambda+\lambda^2+\ldots+\lambda^n
\end{equation*}
We observe that if $\psi \in C^n_\lambda(\mathcal{C})=Ker(1-\lambda)$, then $A\psi=(n+1)\psi$. From the relation
\begin{equation*}
(1-\lambda)(1+2\lambda+3\lambda^2+\cdots+(n+1)\lambda^n)=A-(n+1)\cdot 1
\end{equation*}
it is immediate that $Ker(A) \subseteq Im(1-\lambda)$.

\smallskip
Let $B_0: CN^{n+1}(\mathcal{C}) \longrightarrow CN^n(\mathcal{C})$ be the map defined as follows:
\begin{equation*}
(B_0\phi)(f^0 \otimes \ldots \otimes f^n):=\phi(id_{X_0} \otimes f^0 \otimes \ldots \otimes f^n)-(-1)^{n+1}\phi(f^0 \otimes \ldots \otimes f^n \otimes id_{X_0})
\end{equation*}
for any $f^0 \otimes f^1 \otimes \ldots \otimes f^{n} \in Hom_{\mathcal{C}}(X_1,X_0) \otimes Hom_{\mathcal{C}}(X_2,X_1) \otimes \ldots \otimes Hom_{\mathcal{C}}(X_0,X_{n})$.  We now set
\begin{equation*}
B:=AB_0:  CN^{n+1}(\mathcal{C}) \longrightarrow CN^n(\mathcal{C})
\end{equation*}

\begin{lemma}\label{5.3g} We have

(1) $bA=Ab'$.

\smallskip
(2) $bB+Bb=0$.
\end{lemma}
\begin{proof}
(1) This follows from the general fact that the dual $CN^\bullet(\mathcal C)$  of the cyclic nerve of $\mathcal C$ is a cocyclic module (see, for instance, \cite[$\S$ 2.5]{Loday}).

\smallskip
(2) For any $f^0 \otimes f^1 \otimes \ldots \otimes f^{n} \in Hom_{\mathcal{C}}(X_1,X_0) \otimes Hom_{\mathcal{C}}(X_2,X_1) \otimes \ldots \otimes Hom_{\mathcal{C}}(X_0,X_{n})$ and $\phi \in CN^n{\mathcal{C}}$, we have 
\begin{equation*}
\begin{array}{lll}
&(B_0b\phi)(f^0\otimes \ldots \otimes f^{n})\\
&\quad =(b\phi)(id_{X_0} \otimes f^0 \otimes \ldots \otimes f^n)-(-1)^{n+1}(b\phi)(f^0 \otimes \ldots \otimes f^n \otimes id_{X_0})\\
& \quad= \phi(f^0 \otimes \ldots \otimes f^{n}) + \sum\limits_{i=0}^{n-1}(-1)^{i+1} \phi(id_{X_0} \otimes f^0 \otimes \ldots \otimes f^if^{i+1} \otimes \ldots \otimes f^{n}) +(-1)^{n+1}  \phi(f^n \otimes f^0 \otimes \ldots \otimes f^{n-1}) \\
& \qquad -(-1)^{n+1}\left(\sum\limits_{i=0}^{n-1}(-1)^{i} \phi(f^0 \otimes \ldots \otimes f^if^{i+1} \otimes \ldots \otimes f^{n} \otimes id_{X_0})  \right)
\end{array}
\end{equation*}
On the other hand,
\begin{equation*}
\begin{array}{lll}
&(b'B_0\phi)(f^0\otimes \ldots \otimes f^{n})\\
& \quad =\sum\limits_{i=0}^{n-1}(-1)^{i} \phi(id_{X_0} \otimes f^0 \otimes \ldots \otimes f^if^{i+1} \otimes \ldots \otimes f^{n})-(-1)^n\sum\limits_{i=0}^{n-1}(-1)^{i} \phi(f^0 \otimes \ldots \otimes f^if^{i+1} \otimes \ldots \otimes f^{n} \otimes id_{X_0})
\end{array}
\end{equation*}
Thus, we obtain
\begin{equation*}
(B_0b+b'B_0)(\phi)(f^0\otimes \ldots \otimes f^{n})=\phi(f^0 \otimes \ldots \otimes f^{n})+(-1)^{n+1} \phi(f^n \otimes f^0 \otimes \ldots \otimes f^{n-1})
\end{equation*}
Therefore,
\begin{equation}\label{5.1ln}
(B_0b+b'B_0)(\phi)=\phi-\lambda \phi
\end{equation}
Now, by applying the operator $A$ to both sides of \eqref{5.1ln}, we have
\begin{equation*}
AB_0b+Ab'B_0=0
\end{equation*}
The result now follows by part \it(1).
\end{proof}

\begin{prop}\label{Prop4.3gf}
The image of the map $B:CN^{n+1}(\mathcal{C}) \longrightarrow CN^{n}(\mathcal{C})$ is
$C^n_\lambda(\mathcal{C})$.
\end{prop}
\begin{proof}
Let $\phi \in C^n_\lambda(\mathcal{C})$ and let $R:=\bigoplus\limits_{X,Y \in Ob(\mathcal{C})}Hom(X,Y)$. Then $A$ is an algebra with mutiplication given by composition wherever possible and $0$ otherwise. We choose a linear map $\eta: A \longrightarrow \mathbb{C}$ such that 
\begin{equation*}
\begin{array}{ll}
\eta(f)=0 &\qquad \text{for}~ f \in Hom_\mathcal{C}(X,Y), ~X \neq Y\\
\eta(id_X)=1 & \qquad \forall X \in Ob(\mathcal{C})
\end{array}
\end{equation*}

We now define $\psi \in CN^{n+1}(\mathcal{C})$ by setting
\begin{equation*}
\begin{array}{ll}
\psi(f^0\otimes \ldots \otimes f^{n+1}):=&\eta(f^0)\phi(f^1\otimes \ldots \otimes f^{n+1})+\\
&\quad (-1)^n \left(\phi\left(f^0 \otimes f^1 \otimes \ldots \otimes f^{n}\right)\eta(f^{n+1}) -\eta(f^0)\phi\left(id_{X_1} \otimes f^1 \otimes \ldots \otimes f^{n}\right)\eta(f^{n+1})\right)
\end{array}
\end{equation*}
for any $f^0 \otimes f^1 \otimes \ldots \otimes f^{n+1} \in Hom_{\mathcal{C}}(X_1,X_0) \otimes Hom_{\mathcal{C}}(X_2,X_1) \otimes \ldots \otimes Hom_{\mathcal{C}}(X_0,X_{n+1})$. 
We observe that if the tuple $(f^1, \ldots, f^{n+1})$ is not cyclically composable, i.e., $X_0 \neq X_1$, then the first term vanishes as $\eta(f^0)=0$. Similarly, if the tuple $(f^0, \ldots, f^{n})$ is not cyclically composable, i.e., $X_{n+1} \neq X_0$, then the second term vanishes. For the last term, $\eta(f^0)$ and $\eta(f^{n+1})$ will be non zero only if $X_1=X_0$ and $X_0=X_{n+1}$ which means that $X_{n+1}=X_1$ and the tuple $(id_{X_1},f^1, \ldots, f^{n})$ is cyclically composable.

Then, for   any $g^0 \otimes g^1 \otimes \ldots \otimes g^{n} \in Hom_{\mathcal{C}}(Y_1,Y_0) \otimes Hom_{\mathcal{C}}(Y_2,Y_1) \otimes \ldots \otimes Hom_{\mathcal{C}}(Y_0,Y_{n})$,
we have
\begin{equation*}
\begin{array}{ll}
\psi(id_{Y_0} \otimes g^0 \otimes \ldots \otimes g^{n})&=\eta(id_{Y_0})\phi(g^0\otimes \ldots \otimes g^{n})+(-1)^n \big(\phi\left(id_{Y_0} \otimes g^0 \otimes \ldots \otimes g^{n-1}\right)\eta(g^{n}) \\
& \quad -\phi\left(\eta(id_{Y_0})id_{Y_0} \otimes g^0 \otimes \ldots \otimes g^{n-1}\right)\eta(g^{n})\big)\\
&=\phi(g^0\otimes \ldots \otimes g^{n})
\end{array}
\end{equation*}
Also
\begin{equation*}
\begin{array}{ll}
\psi(g^0 \otimes \ldots \otimes g^{n} \otimes id_{Y_0})&=\eta(g^0)\phi(g^1\otimes \ldots \otimes g^{n} \otimes id_{Y_0})+(-1)^n \big(\phi\left(g^0 \otimes \ldots \otimes g^{n}\right)\eta(id_{Y_0})\\
&\quad -\phi\left(\eta(g^0)id_{Y_1} \otimes g^1 \otimes \ldots \otimes g^{n}\right)\eta(id_{Y_0})\big)\\
& =(-1)^n\phi\left(g^0 \otimes \ldots \otimes g^{n}\right) \\
\end{array}
\end{equation*} where the second equality follows from the fact that $\phi\in C_\lambda^n(\mathcal C)$ and that $\eta(g^0)=0$ whenever $Y_1\ne Y_0$. 
Thus,
\begin{equation*}
\begin{array}{ll}
(B_0\psi)(g^0 \otimes \ldots \otimes g^n)&=\psi(id_{Y_0} \otimes g^0 \otimes \ldots \otimes g^n)-(-1)^{n+1}\psi(g^0 \otimes \ldots \otimes g^n \otimes id_{Y_0})\\
&=2\phi(g^0 \otimes \ldots \otimes g^n)
\end{array}
\end{equation*}
Since $\phi \in Ker(1-\lambda)$, we now have $B\psi=2A\phi=2(n+1)\phi$. Thus, $\phi \in Im(B)$. Conversely, let $\phi \in Im(B)$. Then, $\phi=B\psi$ for some $\psi \in CN^{n+1}(\mathcal{C})$. Using the fact that $(1-\lambda)A=0$, we have
\begin{equation*}
\begin{array}{c}
(1-\lambda)(\phi)=(1-\lambda)(B\psi)=((1-\lambda)AB_0)\psi=0\\
\end{array} 
\end{equation*}
This proves the result.
\end{proof}

\begin{prop}\label{prop5.5h}
Let $\psi \in CN^n(\mathcal{C})$ be such that $b\psi \in C^{n+1}_\lambda(\mathcal{C})$. Then,

\smallskip
(1) $B\psi \in Z^{n-1}_\lambda(\mathcal{C})$ i.e., $b(B\psi)=0$ and $(1-\lambda)(B\psi)=0$.

\smallskip
(2) $S(B\psi)=n(n+1)b\psi$ in $H^{n+1}_\lambda(\mathcal{C})$.
\end{prop}
\begin{proof}
{\it(1)} We know that $(1-\lambda)(B \psi)=(1-\lambda)(AB_0)(\psi)=0$. Further, for any $\phi \in Ker(1-\lambda)$, we have $B_0\phi=0$. Therefore, it follows that $bB\psi=-Bb\psi=-AB_0b\psi=0$.

\smallskip
{\it(2)} We have to show that $SB\psi-n(n+1)b\psi=b\zeta$ for some $\zeta \in C^{n}_\lambda(\mathcal{C})$. We set $\phi=B\psi$. Then, $\phi$ is the character of an $(n-1)$-dimensional cycle $(\mathcal{S},\hat{\partial},\hat{T},\rho)$ over $\mathcal{C}$.  By Proposition \ref{Sbound}, we have
$S\phi=b\psi'$, where $\psi' \in CN^n(\mathcal{C})$ is given by
\begin{equation*}
\psi'(f^0 \otimes \ldots \otimes f^{n})=\sum\limits_{j=1}^n (-1)^{j-1}~ \hat{T}\left(f^0\hat{\partial}f^1 \ldots  \hat{\partial}f^{j-1} f^{j} \hat{\partial}f^{j+1} \ldots \hat{\partial}f^n\right)
\end{equation*}
Suppose we have $\psi'' \in CN^n(\mathcal{C})$ such that $\psi''-\psi \in B^n(\mathcal{C})$ and   $\zeta=\psi'-n(n+1)\psi'' \in C^{n}_\lambda (\mathcal{C})$. This would give
\begin{equation*}
b\zeta=b\psi'-n(n+1)b \psi''=SB\psi-n(n+1)b\psi
\end{equation*}
We set $\theta:=B_0\psi$, $\theta':=\frac{1}{n}\phi$ and $\theta'':=\theta-\theta'\in CN^{n-1}(\mathcal C)$. Since $B\psi \in Z^{n-1}_\lambda(\mathcal{C})$, we have
\begin{equation*}
A\theta''=AB_0\psi-\frac{1}{n}A\phi=B\psi-\frac{1}{n}AB\psi=B\psi-\frac{1}{n}nB\psi=0
\end{equation*}
Since $Ker(A) \subseteq Im(1-\lambda)$, we have $\theta''=(1-\lambda)(\psi_1)$ for some $\psi_1 \in CN^{n-1}(\mathcal{C})$. We take $\psi''=\psi-b\psi_1$. We now show that $(1-\lambda)(\zeta)=0$, i.e, $(1-\lambda)(\psi')=n(n+1)(1-\lambda)(\psi'')$ where $\zeta=\psi'-n(n+1)\psi''$.
We see that
\begin{equation*}
(\tau_n\psi')(f^0 \otimes \ldots \otimes f^{n})=\psi'(f^n \otimes f^0 \otimes \ldots \otimes f^{n-1})=\sum\limits_{j=0}^{n-1} (-1)^{j}~ \hat{T}\left(\hat{\partial}f^0\hat{\partial}f^1 \ldots  \hat{\partial}f^{j-1} f^{j} \hat{\partial}f^{j+1} \ldots \hat{\partial}f^{n-1}f^n\right)
\end{equation*} where we have used the fact that $\hat{T}$ is a graded trace. 
For $1\leq j\leq n-1$, we now set
\begin{equation*}
\omega_j:=f^0(\hat{\partial}f^1 \ldots  \hat{\partial}f^{j-1}) f^{j} (\hat{\partial}f^{j+1} \ldots \hat{\partial}f^{n-1})f^n
\end{equation*}
Then,
\begin{equation*}
\begin{array}{ll}
\hat{\partial}\omega_j&=(\hat{\partial}f^0\hat{\partial}f^1 \ldots  \hat{\partial}f^{j-1}) f^{j} (\hat{\partial}f^{j+1} \ldots \hat{\partial}f^{n-1})f^n+(-1)^{j-1}f^0(\hat{\partial}f^1 \ldots  \hat{\partial}f^{j-1} \hat{\partial}f^{j} \hat{\partial}f^{j+1} \ldots \hat{\partial}f^{n-1})f^n +\\
&\quad (-1)^n f^0(\hat{\partial}f^1 \ldots  \hat{\partial}f^{j-1}) f^{j} (\hat{\partial}f^{j+1} \ldots \hat{\partial}f^{n-1} \hat{\partial}f^n)
\end{array}
\end{equation*}
Thus,
\begin{equation*}
\begin{array}{ll}
0=\hat{T}(\hat{\partial}\omega_j)&=\hat{T}\left(\hat{\partial}f^0\hat{\partial}f^1 \ldots  \hat{\partial}f^{j-1} f^{j} \hat{\partial}f^{j+1} \ldots \hat{\partial}f^{n-1}f^n\right)+(-1)^{j-1}\hat{T}\left(f^0\hat{\partial}f^1 \ldots  \hat{\partial}f^{j-1} \hat{\partial}f^{j} \hat{\partial}f^{j+1} \ldots \hat{\partial}f^{n-1}f^n\right) +\\
&\quad (-1)^n \hat{T}\left(f^0\hat{\partial}f^1 \ldots  \hat{\partial}f^{j-1} f^{j} \hat{\partial}f^{j+1} \ldots \hat{\partial}f^{n-1} \hat{\partial}f^n\right)
\end{array}
\end{equation*}
Therefore,
\begin{equation*}
\begin{array}{ll}
&(1-\lambda)(\psi')(f^0 \otimes \ldots \otimes f^{n})\\
&=-\sum\limits_{j=1}^n (-1)^{j}~ \hat{T}\left(f^0\hat{\partial}f^1 \ldots  \hat{\partial}f^{j-1} f^{j} \hat{\partial}f^{j+1} \ldots \hat{\partial}f^n\right)-
(-1)^n \sum\limits_{j=0}^{n-1} (-1)^{j}~ \hat{T}\left(\hat{\partial}f^0\hat{\partial}f^1 \ldots  \hat{\partial}f^{j-1} f^{j} \hat{\partial}f^{j+1} \ldots \hat{\partial}f^{n-1}f^n\right)\\
&=-\big((-1)^n\hat{T}\left(f^0\hat{\partial}f^1 \ldots   \hat{\partial}f^{n-1}f^n\right)+(-1)^n \hat{T}\left(f^0\hat{\partial}f^1 \ldots  \hat{\partial}f^{n-1}f^n\right) +\\
&\quad \sum\limits_{j=1}^{n-1} (-1)^{j}~\left(\hat{T}\left(f^0\hat{\partial}f^1 \ldots  \hat{\partial}f^{j-1} f^{j} \hat{\partial}f^{j+1} \ldots \hat{\partial}f^n\right) +(-1)^n \hat{T}\left(\hat{\partial}f^0\hat{\partial}f^1 \ldots  \hat{\partial}f^{j-1} f^{j} \hat{\partial}f^{j+1} \ldots \hat{\partial}f^{n-1}f^n\right)\right)\big)\\
&=(-1)^{n+1}~(n+1)\hat{T}(f^nf^0\hat{\partial}f^1 \ldots   \hat{\partial}f^{n-1})=(-1)^{n+1} ~(n+1)\phi(f^nf^0\otimes f^1 \otimes \ldots  \otimes f^{n-1})
\end{array}
\end{equation*}
Hence,
\begin{equation}\label{1c}
(1-\lambda)(\psi')(f^0 \otimes \ldots \otimes f^{n})=(-1)^{n+1} ~(n+1)\phi(f^nf^0\otimes f^1 \otimes \ldots  \otimes f^{n-1})
\end{equation}
On the other hand, using  the definition of $\psi''$ and the fact that $(1-\lambda)b=b'(1-\lambda)$, we have
\begin{equation*}
(1-\lambda)(\psi'')=(1-\lambda)(\psi)-(1-\lambda)(b\psi_1)=(1-\lambda)(\psi)-b'(1-\lambda)(\psi_1)=(1-\lambda)(\psi)-b'\theta''
\end{equation*}
Since $b\psi \in C^{n+1}_\lambda(\mathcal{C})$, we have from \eqref{5.1ln} that $(1-\lambda)(\psi)=(B_0b+b'B_0)(\psi)=b'B_0\psi=b'\theta=b'\theta'+b'\theta''$. Hence, 
\begin{equation*}
(1-\lambda)(\psi'')=b'\theta'=\frac{1}{n}~b'\phi
\end{equation*}
Since $\phi=B\psi \in Z^{n-1}_\lambda(\mathcal{C})$, $b\phi=0$ and therefore
\begin{equation}\label{1d}
(1-\lambda)(\psi'')(f^0 \otimes \ldots \otimes f^{n})=\frac{1}{n}~(b'\phi)(f^0 \otimes \ldots \otimes f^{n})= \frac{1}{n} ~(-1)^{n-1}\phi(f^nf^0 \otimes f^1 \otimes \ldots \otimes f^{n-1})
\end{equation}
The result now follows by comparing \eqref{1c} and \eqref{1d}.
\end{proof}

\begin{prop}\label{Conlemma}
Let $\mathcal C$ be a small $\mathbb{C}$-category  and $\{\mathscr{H}_t:\mathcal{C} \longrightarrow SHilb_2\}_{t \in [0,1]}$ be a family of functors such that
for each $X\in Ob(\mathcal C)$, we have $\mathscr H_t(X)=\mathscr H_{t'}(X)=\mathscr H(X)$ for all $t$, $t'\in [0,1]$. 
For $t \in [0,1]$, we set $\widehat{\mathscr{H}}_t:=i \circ \mathscr{H}_t$, where $i:SHilb_2 \longrightarrow SHilb_{\mathbb{Z}_2}$ is the functor as defined in \eqref{i}. 

\smallskip
Let $\mathcal F$ be the family of operators 
\begin{equation}
\mathcal F=\left\{i(F(\mathscr H(X))=\begin{pmatrix} 0 & 1 \\ 1 & 0 \\
\end{pmatrix}:\mathscr H(X)\oplus\mathscr H(X)\longrightarrow \mathscr H(X)\oplus \mathscr H(X)\right\}_{X\in Ob(\mathcal C)}
\end{equation}

Let $p=2m$ be an even integer. We assume that 

(1)  for each $f \in Hom_\mathcal{C}(X,Y)$, the association  $ t \mapsto [\mathcal{F},\widehat{\mathscr{H}}_t(f)]$ is a continuous map
\begin{equation*}
\zeta_f:[0,1] \longrightarrow \mathcal{B}^p(\mathscr H(X)\oplus \mathscr H(X), \mathscr H(Y)\oplus \mathscr H(Y)) \qquad t \mapsto [\mathcal{F},\widehat{\mathscr{H}}_t(f)]
\end{equation*}

\smallskip
(2) for each $f \in Hom_\mathcal{C}(X,Y)$, the association 
\begin{equation*}
p_f:[0,1] \longrightarrow SHilb_{\mathbb{Z}_2}(\widehat{\mathscr H}_t(X),\widehat{\mathscr H}_t(Y))\qquad t \mapsto \widehat{\mathscr{H}}_t(f)
\end{equation*}
is piecewise strongly $C^1$. 

\smallskip
Let $(\widehat{\mathscr{H}}_t,\mathcal{F})$ be the corresponding $p$-summable Fredholm modules over $\mathcal{C}$. Then, the class in $H^{p+2}_\lambda(\mathcal{C})$ of the $(p+2)$-dimensional character of the Fredholm module $(\widehat{\mathscr{H}}_t,\mathcal{F})$ is independent of $t$.
\end{prop}

\begin{proof}
For any $t \in [0,1]$, let $\phi_{t}$ be the $p$-dimensional character of the Fredholm module $(\widehat{\mathscr{H}}_{t},\mathcal{F})$. We will show that $S(\phi_{t_1})=S(\phi_{t_2})$ for any $t_1, t_2 \in [0,1]$.

\smallskip
By assumption, we know that there exists a finite set $R=\{0= r_0\leq r_1 < ... <r_k\leq r_{k+1}=1\}\subseteq [0,1]$ such that $p_f:[0,1] \longrightarrow SHilb_{\mathbb{Z}_2}(\widehat{\mathscr H}_t(X),\widehat{\mathscr H}_t(Y))$ is continuously differentiable in each $[r_i,r_{i+1}]$. By abuse of notation,
we set for each $f \in Hom_\mathcal{C}(X,Y)$:
\begin{equation}
\delta_t(f):=p_f'(t)\in SHilb_{\mathbb{Z}_2}(\widehat{\mathscr H}_t(X),\widehat{\mathscr H}_t(Y))
\end{equation} Here, it is understood that if $t=r_i$ for some $1\leq i\leq k$, we use the right hand derivative when $r_i$ is treated as a point
of $[r_i,r_{i+1}]$ and the left hand derivative when $r_i$ is treated as a point of $[r_{i-1},r_i]$. 

\smallskip
Using Lemma \ref{5.2x},  we know that 
\begin{equation}\label{nv2}
\delta_t(fg)=\widehat{\mathscr{H}}_t(f) \circ \delta_t(g) + \delta_t(f) \circ \widehat{\mathscr{H}}_t
(g)
\end{equation}
for any $t \in [0,1]$ and for any pair of composable morphisms $f$ and $g$ in $\mathcal{C}$.

\smallskip
For any $t \in [0,1]$ and $1 \leq j \leq p+1$, we set
\begin{equation*}
\psi_t^j(f^0 \otimes \ldots \otimes f^{p+1}):=Tr\left(\epsilon \widehat{\mathscr{H}}_t(f^0)[\mathcal{F}, \widehat{\mathscr{H}}_t(f^1)]\ldots [\mathcal{F}, \widehat{\mathscr{H}}_t(f^{j-1})]\delta_t(f^j)[\mathcal{F}, \widehat{\mathscr{H}}_t(f^{j+1})] \ldots [\mathcal{F}, \widehat{\mathscr{H}}_t(f^{p+1})]\right)
\end{equation*}
Using the expression in \eqref{nv2} and the fact that $\epsilon \widehat{\mathscr H}(f)=\widehat{\mathscr H}(f)\epsilon$ for any morphism $f \in \mathcal{C}$,  it may be easily verified that $b\psi_t^j=0$. For example, when $j=1$,  we have (suppressing the  functor $\widehat{\mathscr H}$ )
\begin{equation*}
\begin{array}{ll}
&(b\psi_t^1)(f^0 \otimes \ldots \otimes f^{p+2})\\
& \quad =\sum\limits_{i=0}^{p+1} \psi_t^1(f^0 \otimes \ldots f^if^{i+1} \otimes \ldots \otimes f^{p+2}) + \psi_t^j(f^{p+2}f^0 \otimes f^1 \otimes \ldots \otimes f^{p+2})\\
& \quad = Tr\left(\epsilon f^0f^1\delta_t(f^2)[\mathcal{F},f^3]\ldots[\mathcal{F},f^{p+2}]\right)-Tr\left(\epsilon f^0 \delta_t(f^1f^2)[\mathcal{F},f^3]\ldots[\mathcal{F},f^{p+2}]\right)\\
& \qquad + Tr\left(\epsilon f^0 \delta_t(f^1)[\mathcal{F},f^2f^3]\ldots[\mathcal{F},f^{p+2}]\right)-Tr\left(\epsilon f^0 \delta_t(f^1)[\mathcal{F},f^2][\mathcal{F},f^3f^4]\ldots[\mathcal{F},f^{p+2}]\right) + \ldots\\
& \qquad \ldots - Tr\left(\epsilon f^0 \delta_t(f^1)[\mathcal{F},f^2]\ldots[\mathcal{F},f^{p+1}f^{p+2}]\right) + Tr\left(\epsilon f^{p+2}f^0 \delta_t(f^1)[\mathcal{F},f^2][\mathcal{F},f^3f^4]\ldots[\mathcal{F},f^{p+1}]\right) \\
&\quad =0
\end{array}
\end{equation*}
 We then define
\begin{equation*}
\begin{array}{ll}
\psi_t:=\sum\limits_{j=0}^{p+1}  (-1)^{j-1} \psi_t^j
\end{array}
\end{equation*}
We have  $b\psi_t=0$.

For fixed $f$, it follows from the compactness of $[0,1]$ and the assumptions (1) and (2) that the families $\{ \widehat{\mathscr{H}}_t(f) \}_{t \in [0,1]}$, 
$\{p_f(t)\}_{t\in [0,1]}$ and 
$\{\delta_t(f)\}_{t \in [0,1]}$ are uniformly bounded. For the sake of simplicity, we assume that there is only a single point $r\in R$ such that 
$t_1\leq r\leq t_2$. Then, we form  $\psi \in CN^{p+1}(\mathcal{C})$ by setting
\begin{equation*}
\psi(f^0 \otimes \ldots \otimes f^{p+1}):=\int_{t_1}^{r}\psi_t(f^0 \otimes \ldots \otimes f^{p+1})dt+\int_{r}^{t_2}\psi_t(f^0 \otimes \ldots \otimes f^{p+1})dt
\end{equation*}
We now have
\begin{equation*}
\begin{array}{ll}
&\psi(id_{X_0} \otimes f^0 \otimes \ldots \otimes f^{p})\\
&=\int_{t_1}^{r}\psi_t(id_{X_0} \otimes f^0 \otimes \ldots \otimes f^{p})dt+\int_{r}^{t_2}\psi_t(id_{X_0} \otimes f^0 \otimes \ldots \otimes f^{p})dt\\
&=\int_{t_1}^{r} \big(\sum\limits_{j=0}^{p}(-1)^{j}Tr\left(\epsilon[\mathcal{F}, \widehat{\mathscr{H}}_t(f^0)]\ldots [\mathcal{F}, \widehat{\mathscr{H}}_t(f^{j-1})]\delta_t(f^j)[\mathcal{F}, \widehat{\mathscr{H}}_t(f^{j+1})] \ldots [\mathcal{F}, \widehat{\mathscr{H}}_t(f^{p})]\right)\big)dt\\&\textrm{ }+\int_{r}^{t_2} \big(\sum\limits_{j=0}^{p}(-1)^{j}Tr\left(\epsilon[\mathcal{F}, \widehat{\mathscr{H}}_t(f^0)]\ldots [\mathcal{F}, \widehat{\mathscr{H}}_t(f^{j-1})]\delta_t(f^j)[\mathcal{F}, \widehat{\mathscr{H}}_t(f^{j+1})] \ldots [\mathcal{F}, \widehat{\mathscr{H}}_t(f^{p})]\right)\big)dt\\
\end{array}
\end{equation*}
Let $\phi:[0,1] \longrightarrow Z^p_\lambda(\mathcal{C})$ be the map given by $t \mapsto \phi_t$. We now claim that 
\begin{equation*}
\psi(id_{X_0} \otimes f^0 \otimes \ldots \otimes f^{p})=\int_{t_1}^{r}\phi'(t)(f^0 \otimes \ldots \otimes f^p) ~dt+\int_{r}^{t_2}\phi'(t)(f^0 \otimes \ldots \otimes f^p) ~dt
\end{equation*}
 Indeed, we have
\begin{equation*}
\begin{array}{ll}
\phi'(t)(f^0 \otimes \ldots \otimes f^p)&=\lim\limits_{s \to 0}\frac{1}{s}(\phi_{t+s}-\phi_t)(f^0 \otimes \ldots \otimes f^p)\\
&=\lim\limits_{s \to 0}\big(Tr\big(\epsilon \frac{1}{s}\left( \widehat{\mathscr{H}}_{t+s}(f^0)-\widehat{\mathscr{H}}_t(f^0)\right)  [\mathcal{F}, \widehat{\mathscr{H}}_{t+s}(f^{1})] \ldots  [\mathcal{F}, \widehat{\mathscr{H}}_{t+s}(f^{p})]\big)\\
& \quad +Tr\big(\epsilon \widehat{\mathscr{H}}_{t}(f^0) [\mathcal{F},\frac{1}{s}\left(\widehat{\mathscr{H}}_{t+s}(f^1)-\widehat{\mathscr{H}}_t(f^1)\right)]  [\mathcal{F}, \widehat{\mathscr{H}}_{t+s}(f^{2})] \ldots  [\mathcal{F}, \widehat{\mathscr{H}}_{t+s}(f^{p})]\big)+ \ldots\\
& \quad +Tr\big(\epsilon \widehat{\mathscr{H}}_{t}(f^0)[\mathcal{F}, \widehat{\mathscr{H}}_{t}(f^1) ] \ldots [\mathcal{F},\frac{1}{s}\left(\widehat{\mathscr{H}}_{t+s}(f^p)-\widehat{\mathscr{H}}_t(f^p)\right)]\big)\big)
\end{array}
\end{equation*}
By {\it(1)}, we know that the association $t \mapsto [\mathcal{F},\widehat{\mathscr{H}}_{t}(f)]$ is a continuous map for each morphism $f \in \mathcal{C}$. Therefore, we have
\begin{equation*}
\begin{array}{ll}
\lim\limits_{s \to 0}\big(Tr\big(\epsilon \widehat{\mathscr{H}}_{t}(f^0)[\mathcal{F}, \widehat{\mathscr{H}}_{t}(f^1)]\ldots [\mathcal{F}, \widehat{\mathscr{H}}_{t}(f^{j-1})] [\mathcal{F},\frac{1}{s}\left(\widehat{\mathscr{H}}_{t+s}(f^{j})-\widehat{\mathscr{H}}_t(f^j)\right)] \ldots  [\mathcal{F}, \widehat{\mathscr{H}}_{t+s}(f^{p})]\big)\big)=\\
=\lim\limits_{s \to 0} (-1)^j \big(Tr\big(\epsilon [\mathcal F,\widehat{\mathscr{H}}_{t}(f^0)]\ldots [\mathcal{F}, \widehat{\mathscr{H}}_{t}(f^{j-1})] \frac{1}{s}\left(\widehat{\mathscr{H}}_{t+s}(f^{j})-\widehat{\mathscr{H}}_t(f^j)\right)[\mathcal{F}, \widehat{\mathscr{H}}_{t+s}(f^{j+1})]  \ldots  [\mathcal{F}, \widehat{\mathscr{H}}_{t+s}(f^{p})]\big)\big)\\
= (-1)^j Tr\big(\epsilon [\mathcal F,\widehat{\mathscr{H}}_{t}(f^0)][\mathcal{F}, \widehat{\mathscr{H}}_{t}(f^1)]\ldots [\mathcal{F}, \widehat{\mathscr{H}}_{t}(f^{j-1})] \delta_t(f^j) [\mathcal{F}, \widehat{\mathscr{H}}_{t}(f^{j+1})] \ldots  [\mathcal{F}, \widehat{\mathscr{H}}_{t}(f^{p})]\big)
\end{array}
\end{equation*}
From this, we obtain
\begin{equation*}
\begin{array}{ll}
&\int_{t_1}^{r}\phi'(t) (f^0 \otimes \ldots \otimes f^p) ~dt+\int_{r}^{t_2}\phi'(t)  (f^0 \otimes \ldots \otimes f^p)~dt\\
&=\int_{t_1}^{r} \sum\limits_{j=0}^p (-1)^j Tr\big(\epsilon [\mathcal F,\widehat{\mathscr{H}}_{t}(f^0)][\mathcal{F}, \widehat{\mathscr{H}}_{t}(f^1)]\ldots [\mathcal{F}, \widehat{\mathscr{H}}_{t}(f^{j-1})] \delta_t(f^j)[\mathcal{F}, \widehat{\mathscr{H}}_{t}(f^{j+1})]  \ldots  [\mathcal{F}, \widehat{\mathscr{H}}_{t}(f^{p})]\big)~ dt\\
&\textrm{ }+\int_{r}^{t_2} \sum\limits_{j=0}^p (-1)^j Tr\big(\epsilon [\mathcal F,\widehat{\mathscr{H}}_{t}(f^0)][\mathcal{F}, \widehat{\mathscr{H}}_{t}(f^1)]\ldots [\mathcal{F}, \widehat{\mathscr{H}}_{t}(f^{j-1})] \delta_t(f^j) [\mathcal{F}, \widehat{\mathscr{H}}_{t}(f^{j+1})] \ldots  [\mathcal{F}, \widehat{\mathscr{H}}_{t}(f^{p})]\big)~ dt\\
&=\psi(id_{X_0} \otimes f^0 \otimes \ldots \otimes f^{p})
\end{array}
\end{equation*}
Hence
\begin{equation*}
\begin{array}{ll}
\psi(id_{X_0} \otimes f^0 \otimes \ldots \otimes f^{p})&=\phi_{t_2}(f^0 \otimes \ldots \otimes f^{p})-\phi_{r}(f^0 \otimes \ldots \otimes f^{p})+\phi_{r}(f^0 \otimes \ldots \otimes f^{p})-\phi_{t_1}(f^0 \otimes \ldots \otimes f^{p})\\
&=\phi_{t_2}(f^0 \otimes \ldots \otimes f^{p})-\phi_{t_1}(f^0 \otimes \ldots \otimes f^{p})\\
\end{array}
\end{equation*}
Since $\psi(f^0 \otimes \ldots \otimes f^{p} \otimes id_{X_0} )=0$, we now have
\begin{equation*}
\begin{array}{ll}
(B_0\psi) (f^0 \otimes \ldots \otimes f^{p})&=\psi(id_{X_0} \otimes f^0 \otimes \ldots \otimes f^{p})-\psi(f^0 \otimes \ldots \otimes f^{p} \otimes id_{X_0} )\\
&=(\phi_{t_2}-\phi_{t_1})(f^0 \otimes \ldots \otimes f^{p})
\end{array}
\end{equation*}
Since $b\psi=0$, using Proposition \ref{prop5.5h} and the fact that $\phi_{t_2}-\phi_{t_1} \in Ker(1-\lambda)$, we have
\begin{equation*}
0=S(B\psi)=S(AB_0\psi)=(p+1)S(\phi_{t_2}-\phi_{t_1})
\end{equation*}
This completes the proof.
\end{proof}

\begin{theorem}\label{Thm5.6t}
Let $\mathcal C$ be a small $\mathbb{C}$-category  and $\{\rho_t:\mathcal{C} \longrightarrow SHilb_2\}_{t \in [0,1]}$ be a family of functors such that
for each $X\in Ob(\mathcal C)$, we have $\rho_t(X)=\rho_{t'}(X)=\rho(X)$ for all $t$, $t'\in [0,1]$. 
For $t \in [0,1]$, we set $\tilde{\rho}_t:=i \circ \rho_t$, where $i:SHilb_2 \longrightarrow SHilb_{\mathbb{Z}_2}$ is the functor as defined in \eqref{i}.  Further, for each $t \in [0,1]$,
let 
\begin{equation}
\mathcal{F}_t:=\left\lbrace \mathcal F_t(X):=\begin{pmatrix} 0 & \mathcal{Q}_t(X) \\ \mathcal{P}_t(X) & 0 \\
\end{pmatrix}:\rho(X)\oplus\rho (X)\longrightarrow \rho(X)\oplus \rho(X) \right\rbrace_{X\in Ob(\mathcal C)}
\end{equation}
with $\mathcal{P}_t(X)=\mathcal{Q}_t^{-1}(X)$ be such that $(\tilde{\rho}_t,\mathcal{F}_t)$ is a $p$-summable Fredholm module over the category $\mathcal{C}$.
We further assume that for some even integer $p$ and for any $f \in Hom_\mathcal{C}(X,Y)$,

\smallskip
(1) $t \mapsto \tilde{\rho}_t^+(f)-\mathcal{Q}_t\tilde{\rho}_t^-(f)\mathcal{P}_t$ is a continuous map from $[0,1]$ to $\mathcal{B}^p({\rho}(X),{\rho}(Y))$.

\smallskip
(2) $t \mapsto \tilde{\rho}_t^+(f)$ and $t \mapsto \mathcal{Q}_t\tilde{\rho}_t^-(f)\mathcal{P}_t$ are piecewise strongly $C^1$ maps from $[0,1]$ to $SHilb(\rho(X),\rho(Y))$.

\smallskip
Then, $\text{ch}^{p+2}(\tilde{\rho}_t,\mathcal{F}_t) \in H^{p+2}(\mathcal{C})$ is independent of $t \in [0,1]$.
\end{theorem}
\begin{proof}
For each $t \in [0,1]$, we set $\mathcal{T}_t:=\begin{pmatrix} 1 & 0 \\ 0 & \mathcal{Q}_t \\
\end{pmatrix}$. Then, $\mathcal{T}_t^{-1}=\begin{pmatrix} 1 & 0 \\ 0 & \mathcal{P}_t \\
\end{pmatrix}$ and $\mathcal{F}'_t:=\mathcal{T}_t\mathcal{F}_t\mathcal{T}_t^{-1}=\begin{pmatrix} 0 & 1\\ 1& 0 \\
\end{pmatrix}.$

\smallskip
For each $t \in [0,1]$, we also define a functor $\widehat{\mathscr H}_t:\mathcal{C} \longrightarrow SHilb_{\mathbb{Z}_2}$ given by
\begin{equation*}
\widehat{\mathscr H}_t(X):=\widehat{\rho}(X) \qquad \widehat{\mathscr H}_t(f):=\mathcal{T}_t\widehat{\rho}_t(f)\mathcal{T}_t^{-1}
\end{equation*}
Then, we have
\begin{equation*}
[\mathcal{F}'_t,\widehat{\mathscr H}_t(f)]=\begin{pmatrix} 0 & \mathcal{Q}_t\widehat{\rho_t}^-(f)\mathcal{P}_t -\widehat{\rho}_t^+(f)\\ \widehat{\rho}_t^+(f)- \mathcal{Q}_t\widehat{\rho_t}^-(f)\mathcal{P}_t  & 0 
\end{pmatrix}
\end{equation*}
Therefore, using assumption {\it(1)}, we see that the map $t \mapsto [\mathcal{F}',\widehat{\mathscr H}_t(f)]$ from $[0,1]$ to $\mathcal{B}^p(\widehat{\mathscr H}(X),\widehat{\mathscr H}(Y))$ is continuous for each $f \in Hom_\mathcal{C}(X,Y)$. Further,
\begin{equation*}
\widehat{\mathscr H}_t(f)=\mathcal{T}_t\widehat{\rho}_t(f)\mathcal{T}_t^{-1}=\begin{pmatrix} \tilde{\rho}_t^+(f) & 0\\ 0  &  \mathcal{Q}_t\tilde{\rho_t}^-(f)\mathcal{P}_t
\end{pmatrix}
\end{equation*}
Therefore, by applying assumption {\it(2)}, we see that the map $t \mapsto \widehat{\mathscr H}_t(f)$ is piecewise strongly $C^1$. Since trace is invariant under similarity, the result now follows using Proposition \ref{Conlemma}.
\end{proof}

\begin{theorem}
Let $\mathcal C$ be a small $\mathbb{C}$-category  and $\{\rho_t:\mathcal{C} \longrightarrow SHilb_2\}_{t \in [0,1]}$ be a family of functors such that
for each $X\in Ob(\mathcal C)$, we have $\rho_t(X)=\rho_{t'}(X)=\rho(X)$ for all $t$, $t'\in [0,1]$. 
For $t \in [0,1]$, we set $\tilde{\rho}_t:=i \circ \rho_t$, where $i:SHilb_2 \longrightarrow SHilb_{\mathbb{Z}_2}$ is the functor as defined in \eqref{i}.  Further, for each $t \in [0,1]$ and $X\in Ob(\mathcal C)$, 
let 
\begin{equation}
\mathcal F_t(X):=\begin{pmatrix} 0 & \mathcal{Q}_t(X) \\ \mathcal{P}_t(X) & 0 \\
\end{pmatrix}:\rho(X)\oplus\rho (X)\longrightarrow \rho(X)\oplus \rho(X)
\end{equation}
with $\mathcal{Q}_t^{-1}=\mathcal{P}_t$ be such that $(\tilde{\rho}_t,\mathcal{F}_t)$ is a $p$-summable Fredholm module over the category $\mathcal{C}$.
We further assume that for some even integer $p$, we have

\smallskip
(1) For any $f \in Hom_\mathcal{C}(X,Y)$, $t \mapsto \tilde{\rho}_t(f)$ is a strongly $C^1$-map from $[0,1]$ to $SHilb_{\mathbb Z_2}(\tilde{\rho}(X),\tilde{\rho}(Y))$.

\smallskip
(2) For any $X\in \mathcal C$, $t \mapsto \mathcal F_t(X)$   is a strongly $C^1$-map from $[0,1]$ to $SHilb_{\mathbb Z_2}(\tilde{\rho}(X),\tilde{\rho}(X))$.

\smallskip
Then, $\text{ch}^{p+2}(\tilde{\rho}_t,\mathcal{F}_t) \in H^{p+2}(\mathcal{C})$ is independent of $t \in [0,1]$.
\end{theorem}

\begin{proof} By definition, $\tilde{\rho}_t(f)=\begin{pmatrix} \tilde\rho^+(f) & 0 \\ 
0 & \tilde\rho^-(f)\\ \end{pmatrix}$ and $\mathcal F_t(X)=\begin{pmatrix} 0 & \mathcal{Q}_t(X) \\ \mathcal{P}_t(X) & 0 \\
\end{pmatrix}$. As such, it is clear that a system satisfying the assumptions (1) and (2) above also satisfies the assumptions in Theorem \ref{Thm5.6t}. This proves the result.

\end{proof}

\section{Odd Fredholm modules over categories}
In this section, we will introduce odd Fredholm modules over a small $\mathbb{C}$-category $\mathcal{C}$ and show that they lead to odd cyclic cocycles over $\mathcal{C}$. 

\smallskip
We first recall the category $SHilb$ whose objects are separable Hilbert spaces and whose morphisms are bounded linear maps.

\begin{definition}
Let $\mathcal{C}$ be a small $\mathbb{C}$-category and let $p \in [1,\infty)$.  
We consider a pair $(\mathscr{H},\mathcal F)$ where

\smallskip
(1) $\mathscr{H}:\mathcal{C}  \longrightarrow SHilb$ is a $\mathbb{C}$-linear functor and 

\smallskip
(2) $\mathcal{F}:=\{\mathcal{F}_X:\mathscr{H}(X) \longrightarrow \mathscr{H}(X)\}_{X \in Ob(\mathcal{C})}$ is a collection of bounded linear maps such that
$\mathcal{F}_X^2=id_{\mathscr{H}(X)}$ 

\smallskip
Then, we say that the pair $(\mathscr{H},\mathcal F)$ is a $p$-summable odd  Fredholm module over the category $\mathcal{C}$ if
every morphism $f:X \longrightarrow Y$ in $\mathcal{C}$ satisfies
\begin{equation}\label{equi}
[\mathcal{F},f]= \mathcal{F}_Y \circ \mathscr{H}(f) - \mathscr{H}(f) \circ \mathcal{F}_X      \in \mathcal{B}^p\left(\mathscr{H}(X),\mathscr{H}(Y)\right)
\end{equation}
\end{definition}

We consider the DG-semicategory $(\Omega' \mathcal{C},\partial')$ defined in 
Section \ref{evencycl}.

\smallskip
For any $T \in \mathcal{B}(\mathscr H(X), \mathscr H(X))$ such that $[\mathcal{F},T] \in \mathcal{B}^1(\mathscr H(X), \mathscr H(X))$, we define
\begin{equation*}
Tr'_s(T):=\frac{1}{2}~Tr\left(\mathcal F_X[\mathcal F,T]\right)=\frac{1}{2}~Tr\left( \mathcal F_X\partial'(T)\right)=\frac{1}{2}~ Tr\left(\mathcal{F}_X(\mathcal{F}_X \circ T - (-1)^{|T|}~ T \circ \mathcal{F}_X)\right)
\end{equation*}

We now proceed to show that odd summable Fredholm modules over $\mathcal{C}$ correspond to the odd cyclic cocycles over the category $\mathcal{C}$.

\begin{lemma}\label{newlem}
Let $(\mathscr{H},\mathcal{F})$ be a $p$-summable odd Fredholm module over a small $\mathbb{C}$-category $\mathcal{C}$. Then, 

\smallskip
(1) for any $n\geq p-1$, $X$, $Y\in Ob(\mathcal C)$ and $k\geq 0$, we have 
\begin{equation}\label{rt2}
Hom^k_{\Omega'{\mathcal C}}(X,Y)
\subseteq \mathcal B^{(n+1)/k}(\mathscr H(X),\mathscr H(Y))
\end{equation}

\smallskip
(2) for any operator $T \in \mathcal{B}(\mathscr H(X), \mathscr H(X))$ of even degree such that $[\mathcal{F},T] \in \mathcal{B}^1(\mathscr H(X), \mathscr H(X))$, we have $Tr'_s(T)=0$.
\end{lemma}
\begin{proof}
{\it (1)} This follows by H\"older's inequality and the $p$-summability.

\smallskip
{\it (2)} Using the fact that $[\mathcal F,T]$ is trace-class, we have
\begin{equation}
Tr'_s(T)=\frac{1}{2}Tr(\mathcal{F}_X\partial'(T))=-\frac{1}{2}Tr(\partial'(T)\mathcal{F}_X)=-\frac{1}{2}Tr(\mathcal{F}_X\partial'(T))=-Tr'_s(T)
\end{equation}
Hence, $Tr'_s(T)=0$. 
\end{proof}

\begin{prop}\label{Tr}
Let $(\mathscr{H},\mathcal{F})$ be a $p$-summable odd Fredholm module over a small $\mathbb{C}$-category $\mathcal{C}$. Take $2m \geq p$. Then, the collection
\begin{equation}
\widehat{Tr}_s=\{Tr'_s:Hom^{2m-1}_{\Omega'\mathcal{C}}(X,X) \longrightarrow \mathbb C\}_{X\in Ob(\mathcal C)}
\end{equation} defines a closed graded trace of dimension $2m-1$ on the DG-semicategory $(\Omega'\mathcal C,\partial')$. 
\end{prop}

\begin{proof}
Using the $p$-summability and \eqref{rt2},  it may be easily verified that $Tr'_s$ is well defined on $Hom^{2m-1}_{\Omega'\mathcal{C}}(X,X)$.
It is also easy to see that  $Tr'_s$ is closed.  We now prove that the collection $\widehat{Tr}_s$ satisfies the condition in \eqref{trace}.  For  any $T_1\in Hom^{i}_{\Omega'\mathcal{C}}(X,Y)$ and $T_2\in Hom^{j}_{\Omega'\mathcal{C}}(Y,X)$ such that $i+j=2m-1$, using \eqref{commutr} we have
\begin{equation*}
\begin{array}{ll}
2Tr'_s(T_1T_2)&=2Tr(\mathcal{F}_Y\partial'(T_1T_2))=2Tr(\mathcal{F}_Y\partial'(T_1)T_2)+(-1)^i2Tr(\mathcal F_Y T_1\partial'(T_2))\\
&=(-1)^{i+1} 2Tr(\partial'(T_1)\mathcal{F}_X T_2)+(-1)^i2Tr(\partial'(T_2)\mathcal F_Y T_1)\\
&=(-1)^{i+1} 2Tr(\mathcal{F}_X T_2\partial'(T_1))+(-1)^{i+j+1}Tr(\mathcal{F}_X \partial'(T_2)T_1)\\
&=(-1)^{j} 2Tr(\mathcal{F}_X T_2\partial'(T_1))+Tr(\mathcal{F}_X \partial'(T_2)T_1)\\
&=2Tr'_s(T_2T_1)
\end{array}
\end{equation*}
\end{proof}

\begin{theorem}\label{oddcycle}
Let $(\mathscr{H},\mathcal{F})$ be a $p$-summable odd Fredholm module over a small $\mathbb{C}$-category $\mathcal{C}$. Take $2m \geq p$. Then, the tuple $(\Omega' \mathcal{C}, \partial', \widehat{Tr}_s,\rho')$ defines a $(2m-1)$-dimensional cycle over $\mathcal{C}$. Then, the element $\phi^{2m-1} \in CN^{2m-1}(\mathcal{C})$ defined by
\begin{equation*}
\phi^{2m-1}(f^0 \otimes f^1 \otimes \ldots \otimes f^{2m-1}):=Tr'_s\left(f^0[\mathcal F,f^1] \ldots [\mathcal F,f^{2m-1}]\right)
\end{equation*}
for any $f^0 \otimes f^1 \otimes \ldots \otimes f^{2m-1} \in Hom_{\mathcal{C}}(X_1,X) \otimes Hom_{\mathcal{C}}(X_2,X_1) \otimes \ldots \otimes Hom_{\mathcal{C}}(X,X_{2m-1})$ is a cyclic cocycle over $\mathcal{C}$.
\end{theorem}

We will refer to $\phi^{2m-1}$ as the $(2m-1)$-dimensional character associated with the $p$-summable odd Fredholm module $(\mathscr{H},\mathcal F)$ over the category $\mathcal{C}$.

\begin{remark}\label{remf}
Using Lemma \ref{newlem} {\it (2)}, we see that only odd dimensional cyclic cocycles appear in the statement of Theorem \ref{oddcycle} because
\begin{equation*}
\phi^{2m}(f^0 \otimes f^1 \otimes \ldots \otimes f^{2m}):=Tr'_s\left(f^0[\mathcal F,f^1] \ldots [\mathcal F,f^{2m}]\right)=0
\end{equation*}
\end{remark}

\end{document}